\documentclass[11pt]{article}
\usepackage[margin=1in]{geometry}
\usepackage{amsmath,amssymb,amsthm,mathtools}
\usepackage{booktabs,tabularx,array,microtype}
\usepackage{algorithm,algpseudocode}
\usepackage[numbers,sort&compress]{natbib}
\usepackage{graphicx}
\usepackage{hyperref}
\hypersetup{
 colorlinks=true,linkcolor=blue,citecolor=blue,urlcolor=blue,
 pdftitle={Near-Optimal Complexity of Finite-Sum Nonconvex-Strongly-Concave Minimax Optimization},
 pdfauthor={Qihao Zhou}
}
\newtheorem{theorem}{Theorem}
\newtheorem{lemma}{Lemma}
\newtheorem{proposition}{Proposition}

\theoremstyle{definition}
\newtheorem{assumption}{Assumption}
\newtheorem{definition}{Definition}
\theoremstyle{remark}
\newtheorem{remark}{Remark}
\newcommand{\R}{\mathbb R}
\newcommand{\E}{\mathbb E}
\newcommand{\Prob}{\mathbb P}
\newcommand{\norm}[1]{\lVert#1\rVert}
\newcommand{\eps}{\varepsilon}
\newcommand{\tO}{\widetilde{O}}
\newcommand{\tTheta}{\widetilde{\Theta}}
\newcommand{\supp}{\operatorname{supp}}
\newcommand{\Lip}{\operatorname{Lip}}
\newcommand{\Fas}{\mathcal F_{\mathrm{AS}}^0}
\newcommand{\Fis}{\mathcal F_{\mathrm{IS}}^0}
\newcommand{\Bis}{\mathcal B_{\mathrm{IS}}^0}
\newcommand{\N}{\mathfrak N}
\newcommand{\Nexp}{\mathfrak N^{\mathrm{exp}}}
\newcolumntype{Y}{>{\raggedright\arraybackslash}X}
\title{Near-Optimal Complexity of Finite-Sum\\Nonconvex-Strongly-Concave Minimax Optimization}
\author{Qihao Zhou\thanks{Independent Researcher. Email: \href{mailto:qihao_zhou@berkeley.edu}{\texttt{qihao\_zhou@berkeley.edu}}.}}
\date{}

\begin{document}
\maketitle
\begin{abstract}
We characterize, up to logarithmic factors, the minimax expected query complexity of finite-sum nonconvex--strongly-concave optimization under mean-squared averaged smoothness. For randomized zero-respecting incremental first-order algorithms, the complexity is \mbox{$\tTheta(n+\min\{\sqrt n\,\kappa,n^{3/4}\sqrt\kappa\}\,L\Delta\eps^{-2})$} throughout \mbox{$\kappa=L/\mu\ge1$}, under exact dual initialization and for \mbox{$0<\eps^2\le c_0L\Delta$}, where $c_0>0$ is universal. This characterization is our main result. A dual-chain construction provides the lower bound and matches the existing Catalyst upper bound for \mbox{$\kappa\ge\sqrt n$}. For \mbox{$1\le\kappa\le\sqrt n$}, we introduce recursive proximal descent--ascent (RPDA), which retains a recursive gradient estimator across proximal subproblems and attains \mbox{$\tO(n+\sqrt n\,\kappa L\Delta\eps^{-2})$} with a fixed query budget. Both upper bounds guarantee an expected squared primal gradient at most \mbox{$\eps^2$}. Under individual smoothness, we also prove \mbox{$\Omega(n+\min\{\kappa,\sqrt{n\kappa}\}\,L\Delta\eps^{-2})$} for every \mbox{$\kappa\ge1$}. It matches the polynomial upper rate for a bilinear subclass when \mbox{$\kappa\ge n$}; the intermediate-condition-number gap remains open.
\end{abstract}

\section{Introduction}\label{sec:intro}
We consider the finite-sum minimax problem
\begin{equation}\label{eq:problem}
 \min_{x\in\R^{d_x}}\max_{y\in\R^{d_y}}
 f(x,y):=\frac1n\sum_{i=1}^{n}f_i(x,y),
 \qquad \Phi(x):=\max_{y\in\R^{d_y}}f(x,y),
\end{equation}
where $f$ can be nonconvex in $x$ and is $\mu$-strongly concave in $y$. Strong concavity makes the maximizer unique and the primal function $\Phi$ differentiable. We seek a point $\hat x$ satisfying $\E\norm{\nabla\Phi(\hat x)}^2\le\eps^2$, using an incremental first-order oracle (IFO) that returns the gradient of one selected component at each query.

The finite-sum structure permits variance reduction, but the cost of estimating the primal gradient also depends on the inner maximization problem. Under averaged smoothness, \citet{zhang} established a lower bound of $\Omega(n+\sqrt{n\kappa}\,L\Delta\eps^{-2})$ for a linear-span class in the large-condition-number regime. Their Catalyst framework gives the upper bound $\tO((n+n^{3/4}\sqrt\kappa)\,L\Delta\eps^{-2})$ when the initial dual error is zero. Here $\kappa=L/\mu$ and $\Phi(0)-\inf\Phi\le\Delta$. A lower bound of the same order also holds for the proximal incremental oracle class studied by \citet{han}. These results leave a factor of $n^{1/4}$ between the polynomial rates when $\kappa\ge2n$, a gap also recorded as open in Case~2-4 of the ICLR 2025 blog post by \citet{zhanghu2025}.

Our main contribution is a matching lower--upper characterization throughout $\kappa\ge1$: the minimax expected IFO complexity is $\tTheta(n+\min\{\sqrt n\,\kappa,n^{3/4}\sqrt\kappa\}\,L\Delta\eps^{-2})$ for $0<\eps^2\le c_0L\Delta$ (Theorem~\ref{thm:matching}). The characterization concerns randomized zero-respecting algorithms, permits adaptive component selection, and assumes exact dual initialization. The lower bound has no logarithmic loss; the upper bound agrees in every polynomial factor.

The two regimes require different contributions. For $\kappa\ge\sqrt n$, we strengthen the lower bound to match Catalyst. The construction connects consecutive primal coordinates through dual paths, following \citet{li}, and assigns their internal edges to random components. Searching for an edge costs order $n$ queries; the admissible path length is order $\sqrt{\kappa/\sqrt n}$. These costs multiply across the primal chain.

For $1\le\kappa\le\sqrt n$, we introduce recursive proximal descent--ascent (RPDA; Algorithm~\ref{alg:recursiveprox}). It combines proximal regularization with the recursive estimator of \citet{spider}. The estimator is retained across successive subproblems, and a total-movement estimate controls its accumulated error. The resulting $\tO(n+\sqrt n\,\kappa L\Delta\eps^{-2})$ upper bound matches an amplitude-adjusted version of the lower-bound construction. Algorithm~\ref{alg:recursiveprox} has a fixed query budget and includes a final refinement that guarantees stationarity of $\Phi$ itself. The proof is given in Appendix~\ref{app:smallupper}.

Under individual smoothness, we obtain $\Omega(n+\min\{\kappa,\sqrt{n\kappa}\}\,L\Delta\eps^{-2})$ on a bilinear subclass for every $\kappa\ge1$. Combining it with the minimization lower bound of \citet{densehiding} gives a comparison throughout $\kappa\ge1$. RPDA improves the upper bound below $\sqrt n$, while a proximal reduction of existing separable saddle-point methods matches the lower bound for $\kappa\ge n$. The remaining bilinear gap, and the complexity of general nonlinear coupling under individual smoothness, remain open.

Section~\ref{sec:prelim} specifies the assumptions and oracle model. Section~\ref{sec:related} reviews prior work and compares the rates. Section~\ref{sec:results} states the results. Sections~\ref{sec:construction} and~\ref{sec:smallupper} outline the lower-bound construction and the matching upper bound. Section~\ref{sec:experiments} illustrates the two regimes numerically. Detailed proofs are in the appendices.

\section{Preliminaries}\label{sec:prelim}
We use $\norm{\cdot}$ for the Euclidean norm of a vector and the spectral norm of a matrix. For an integer $q\ge1$, let $[q]=\{1,\ldots,q\}$. We write $z=(x,y)$ for the joint variable, $e_k$ for a standard basis vector, and $\supp(v)$ for the set of nonzero coordinates of $v$. The notation $O(\cdot)$, $\Omega(\cdot)$, and $\Theta(\cdot)$ hides universal constants. A tilde additionally suppresses logarithmic factors in problem and accuracy parameters. Thus $\tTheta$ denotes agreement up to logarithms, not an exact constant-factor characterization.

\subsection{Function classes and stationarity}
We consider continuously differentiable components on unconstrained Euclidean domains and distinguish the following two smoothness conditions.
\begin{assumption}[Averaged smoothness, AS]\label{ass:as}
For all $z,z'$,
\begin{equation}\label{eq:as}
 \frac1n\sum_{i=1}^{n}\norm{\nabla f_i(z)-\nabla f_i(z')}^2
 \le L^2\norm{z-z'}^2.
\end{equation}
\end{assumption}
\begin{assumption}[Individual smoothness, IS]\label{ass:is}
For all $i\in[n]$ and all $z,z'$, $\norm{\nabla f_i(z)-\nabla f_i(z')}\le L\norm{z-z'}$.
\end{assumption}
Individual smoothness implies averaged smoothness. The converse need not hold: \eqref{eq:as} allows an individual component to have smoothness of order $\sqrt nL$.

\begin{assumption}[Strong concavity and initialization]\label{ass:init}
For every $x$, the function $f(x,\cdot)$ is $\mu$-strongly concave. The primal function is bounded below and satisfies
\begin{equation}\label{eq:init}
 \Phi(0)-\inf_x\Phi(x)\le\Delta,
 \qquad y^*(0):=\arg\max_y f(0,y)=0.
\end{equation}
\end{assumption}
The exact dual initialization in \eqref{eq:init} makes the lower--upper comparison explicit. It restricts the hard class, so a lower bound proved here also applies when only the primal gap is bounded. It does not justify dropping an initialization term from an upper bound on a larger class.

Strong concavity on $\R^{d_y}$ gives a unique maximizer $y^*(x)$, and $\nabla\Phi(x)=\nabla_x f(x,y^*(x))$. We measure accuracy by
\begin{equation}\label{eq:criterion}
 \E\norm{\nabla\Phi(\hat x)}^2\le\eps^2.
\end{equation}
Let $\Fas(n,L,\mu,\Delta)$ and $\Fis(n,L,\mu,\Delta)$ denote the classes satisfying Assumption~\ref{ass:init} and, respectively, Assumption~\ref{ass:as} or~\ref{ass:is}, over all finite dimensions.

For the individual-smoothness result, we also use the subclass $\Bis\subseteq\Fis$ of functions
\begin{equation}\label{eq:bilinearclass}
 f_i(x,y)=g_i(x)+\langle A_ix,y\rangle-h_i(y),
\end{equation}
where $g_i$ and $h_i$ are $L$-smooth, $h_i$ is $\mu$-strongly convex, and $\norm{A_i}\le L$. The joint individual-smoothness condition is retained. Our dual-path instances have a common separable $g_i$ and quadratic $h_i$.

\subsection{Oracle and algorithm class}
\begin{definition}[Incremental first-order oracle]
Given an index $i\in[n]$ and a point $(x,y)$, an IFO query returns $\nabla f_i(x,y)$. Each component-gradient evaluation counts as one query; computing a full gradient costs $n$ queries.
\end{definition}

\begin{definition}[Zero-respecting algorithm]\label{def:zr}
Write $(g_x^t,g_y^t)=\nabla f_{i_t}(x^t,y^t)$ for the response to query $t$. The first query is at $(0,0)$. For every $t\ge1$, a zero-respecting algorithm satisfies
\begin{equation}\label{eq:zr}
 \supp(x^t)\subseteq\bigcup_{s<t}\supp(g_x^s),\quad
 \supp(y^t)\subseteq\bigcup_{s<t}\supp(g_y^s),\quad
 \supp(\hat x^t)\subseteq\bigcup_{s\le t}\supp(g_x^s),
\end{equation}
where $\hat x^t$ is an output after $t$ queries, and $\hat x^0=0$.
\end{definition}
Indices, query values, and stopping decisions may be arbitrary measurable functions of the transcript and private randomness. There is no bound on the magnitude of exposed coordinates. Ordinary gradient linear-span methods satisfy Definition~\ref{def:zr}; arbitrary changes of basis or unrestricted output vectors need not satisfy it.

For a function class $\mathcal F$, let $\N_\eps(\mathcal F)$ be the smallest fixed query budget of a randomized zero-respecting algorithm that guarantees \eqref{eq:criterion} on every instance in $\mathcal F$. The algorithm may know the class parameters and the dimension. Let $\Nexp_\eps(\mathcal F)$ instead minimize the worst-case expected query count over zero-respecting algorithms with a random stopping time and the same accuracy guarantee. In particular, $\Nexp_\eps\le\N_\eps$.

\section{Related Work and Complexity Comparison}\label{sec:related}
\paragraph{NC--SC lower bounds.}
\citet{zhang,li} establish the deterministic $\sqrt\kappa$ dependence. For finite sums, \citet{zhang} and \citet{han} prove $\Omega(n+\sqrt{n\kappa}\,L\Delta\eps^{-2})$ in their large-condition-number regimes, for linear-span IFO and restricted PIFO classes, respectively. The averaged-smooth comparison is recorded as open in Case~2-4 of \citet{zhanghu2025}. Theorem~\ref{thm:matching} resolves its polynomial gap for randomized zero-respecting IFO algorithms with exact dual initialization; it makes no PIFO claim.

\paragraph{The $n^{3/4}$ dependence and dual paths.}
The factor $n^{3/4}$ already occurs under averaged smoothness for strongly convex finite sums with nonconvex components \cite[Theorem 4.1]{zhougu}, and for unbalanced strongly-convex--strongly-concave problems \cite[Theorem 22]{han}, with matching upper bounds in \citet{luounbalanced}. Our result concerns its NC--SC stationarity counterpart. The dual paths and separable corrections follow \citet[Equations (6)--(10)]{li}. We combine them with the primal chain of \citet{cdhs}, assigning only internal dual edges to random component owners and giving adjacent edges distinct owners. This decomposition yields the product $n\sqrt{\kappa/\sqrt n}$ per primal link.

\paragraph{Variance reduction and proximal methods.}
SREDA has rate $\tO(n+\sqrt n\,\kappa^2L\Delta\eps^{-2})$ when $n\ge\kappa^2$, for expected gradient norm \cite{sreda,zhang}. \citet{yangvr} introduce VR-AGDA under two-sided Polyak--\L{}ojasiewicz (PL) conditions. Under individual smoothness, \citet{chenpl} obtain $O((n+\sqrt n\,\kappa_x\kappa_y^2)\log(1/\eps))$ with SPIDER-GDA, and accelerate the condition-number factor to $\sqrt n\,\kappa_x\kappa_y$ when $\kappa_y\gtrsim\sqrt n$, up to logarithms. Their one-sided PL extension gives stationarity rates $O(n+\sqrt n\,\kappa_y^2L\Delta\eps^{-2})$ and $\tO((n+\sqrt n\,\kappa_y)L\Delta\eps^{-2})$. RPDA instead attains $\tO(n+\sqrt n\,\kappa L\Delta\eps^{-2})$ below $\sqrt n$ under averaged smoothness, with a fixed budget and criterion \eqref{eq:criterion}. It retains the recursive estimator of \citet{spider} across proximal centers and controls its error through total movement. Single-loop alternatives include Smoothed-AGDA, whose deterministic iteration bound is $O(\kappa L\Delta\eps^{-2})$ under one-sided PL \cite{yangsmoothed}.

More recently, \citet[Theorem~2.4]{khanxu2026} analyze SPIDER-based smoothed descent--ascent under extended Kurdyka--\L{}ojasiewicz conditions. At $\theta=1/2$, their finite-sum bound has $O(n+\sqrt n\,\kappa\Delta_{\mathrm{KL}}\eps^{-2})$ dependence after normalizing smoothness, where $\Delta_{\mathrm{KL}}$ is their initial Lyapunov gap: strong concavity permits the KL constant $\mu_{\mathrm{KL}}^2=2\mu$. Their theorem assumes a compact dual domain and controls both game-stationarity residuals (Definition~1.1), so it does not directly give the unconstrained primal-gradient guarantee~\eqref{eq:criterion}. Converting equal residual tolerances to primal-gradient accuracy can introduce a factor $\kappa$ in the tolerance.

\paragraph{Individual smoothness and unrestricted queries.}
The separable saddle-point method of \citet{jst} gives the bilinear upper bound in Table~\ref{tab:comparison-is}; Appendix~\ref{app:upper} supplies the proximal reduction. The minimization lower bound of \citet{densehiding} applies to unrestricted randomized IFO algorithms under individual smoothness. Adding a decoupled quadratic dual variable embeds it in our bilinear class. Appendix~\ref{app:iscomparison} checks the accuracy criterion and expected query convention.

\paragraph{Other oracle models.}
Stochastic lower bounds \cite{stochasticzhou,stochasticzhang} concern fresh oracle samples; nonconvex--concave lower bounds \cite{ncconcave} use different curvature and domain assumptions. Their rates do not directly transfer to the finite-sum classes considered here.

\begin{table}[tb]
\centering
\caption{Component-oracle complexities under averaged smoothness (AS). Upper bounds precede lower bounds. The bold complexities are our matching rates under the accuracy restriction of Theorem~\ref{thm:matching}.}
\label{tab:comparison}
\small
\setlength{\tabcolsep}{4pt}
\renewcommand{\arraystretch}{1.35}
\begin{tabularx}{\textwidth}{@{}>{\raggedright\arraybackslash}p{0.27\textwidth} >{\raggedright\arraybackslash}p{0.20\textwidth} >{\raggedright\arraybackslash$\displaystyle}X<{$}@{}}
\toprule
Result & Oracle/class & \multicolumn{1}{l@{}}{Complexity} \\
\midrule
SREDA \cite{sreda}, $n\ge\kappa^2$
 & IFO & \tO(n+\sqrt n\,\kappa^2 L\Delta\eps^{-2}) \\
Catalyst \cite{zhang}
 & IFO & \tO((n+n^{3/4}\sqrt\kappa)L\Delta\eps^{-2}) \\
\textbf{This paper, Theorem~\ref{thm:matching}}
 & ZR IFO & \boldsymbol{\tO(n+\min\{\sqrt n\,\kappa,n^{3/4}\sqrt\kappa\}L\Delta\eps^{-2})} \\
\midrule[0.3pt]
\citet{zhang}
 & Linear-span IFO & \Omega(n+\sqrt{n\kappa}\,L\Delta\eps^{-2}) \\
\citet{han}, Theorem 25
 & PIFO class & \Omega(n+\sqrt{n\kappa}\,L\Delta\eps^{-2}) \\
\textbf{This paper, Theorem~\ref{thm:as}}
 & Randomized ZR IFO & \boldsymbol{\Omega(n+\min\{\sqrt n\,\kappa,n^{3/4}\sqrt\kappa\}L\Delta\eps^{-2})} \\
\bottomrule
\end{tabularx}
\par\vspace{4pt}
\begin{minipage}{\textwidth}\footnotesize
ZR denotes zero-respecting algorithms. Upper bounds use exact dual initialization. This paper uses RPDA for $\kappa\le\sqrt n$ (fixed budget) and Catalyst \cite{zhang} otherwise (expected queries). Prior lower bounds retain their parameter restrictions; the PIFO class of \citet{han} restricts full-gradient access. The SREDA row uses $n\ge\kappa^2$ and the comparison in \citet[Remark 4.2]{zhang}, with its original domain and expected-norm criterion; our upper bound guarantees \eqref{eq:criterion}.
\end{minipage}
\end{table}

\begin{table}[tb]
\centering
\caption{Component-oracle complexities for the individually smooth bilinear subclass (IS--B), \eqref{eq:bilinearclass}. Upper bounds precede lower bounds. The bold combined rates match up to logarithms for $\kappa\ge n$ and for bounded $\kappa$.}
\label{tab:comparison-is}
\small
\setlength{\tabcolsep}{4pt}
\renewcommand{\arraystretch}{1.35}
\begin{tabularx}{\textwidth}{@{}>{\raggedright\arraybackslash}p{0.33\textwidth} >{\raggedright\arraybackslash}p{0.12\textwidth} >{\raggedright\arraybackslash$\displaystyle}X<{$}@{}}
\toprule
Result & Range & \multicolumn{1}{l@{}}{Complexity} \\
\midrule
Proximal reduction of \citet{jst}
 & all $\kappa$ & \tO((n+\sqrt{n\kappa})L\Delta\eps^{-2}) \\
\textbf{RPDA, Theorem~\ref{thm:smallupper}}
 & $\kappa\le\sqrt n$ & \boldsymbol{\tO(n+\sqrt n\,\kappa L\Delta\eps^{-2})} \\
\textbf{Combined upper bound}
 & all $\kappa$ & \boldsymbol{\tO(n+\min\{\sqrt n\,\kappa,n+\sqrt{n\kappa}\}L\Delta\eps^{-2})} \\
\midrule[0.3pt]
\citet{densehiding}, minimization
 & all $\kappa$ & \Omega(n+\sqrt n\,L\Delta\eps^{-2}) \\
\textbf{This paper, Theorem~\ref{thm:is}}
 & all $\kappa$ & \boldsymbol{\Omega(n+\min\{\kappa,\sqrt{n\kappa}\}L\Delta\eps^{-2})} \\
\textbf{Combined lower bound}
 & all $\kappa$ & \boldsymbol{\Omega(n+\max\{\sqrt n,\min\{\kappa,\sqrt{n\kappa}\}\}L\Delta\eps^{-2})} \\
\bottomrule
\end{tabularx}
\par\vspace{4pt}
\begin{minipage}{\textwidth}\footnotesize
Here ``all $\kappa$'' means $\kappa\ge1$. The comparison uses $n\ge2$, exact dual initialization, criterion \eqref{eq:criterion}, and $0<\eps^2\le c_{\mathrm{IS}}L\Delta$, with a universal $c_{\mathrm{IS}}>0$. Bounds concern randomized ZR IFO algorithms; the minimization lower bound holds for a larger class. Proposition~\ref{prop:iscomparison} proves the combined bounds for expected query counts. RPDA also has a fixed budget. The accuracy restriction absorbs the $1+$ in the proximal reduction.
\end{minipage}
\end{table}

\section{Main Results}\label{sec:results}
Our main result is the following minimax characterization, obtained by matching a lower bound with algorithmic upper bounds throughout the condition-number range.
\begin{theorem}[Near-optimal minimax complexity]\label{thm:matching}
There is a universal constant $c_0>0$ such that, for $n\ge2$, $\kappa=L/\mu\ge1$, and $0<\eps^2\le c_0L\Delta$,
\begin{equation}\label{eq:matchedcomplexity}
 \Nexp_\eps\bigl(\Fas(n,L,\mu,\Delta)\bigr)=
 \begin{cases}
 \tTheta(n+\sqrt n\,\kappa L\Delta\eps^{-2}),&1\le\kappa\le\sqrt n,\\
 \tTheta(n+n^{3/4}\sqrt\kappa\,L\Delta\eps^{-2}),&\kappa\ge\sqrt n.
 \end{cases}
\end{equation}
For $1\le\kappa\le\sqrt n$, the same statement holds for the fixed-budget complexity $\N_\eps$. The lower bound in \eqref{eq:matchedcomplexity} is $\Omega$ without logarithmic factors.
\end{theorem}
\begin{proof}
Theorem~\ref{thm:as} gives the lower bound in both regimes. For $\kappa\le\sqrt n$, RPDA attains the matching upper bound with a fixed query budget by Theorem~\ref{thm:smallupper}. For $\kappa\ge\sqrt n$, the expected-query Catalyst bound \cite[Corollary 4.1 and Section 4.2]{zhang} matches the lower bound because $n^{3/4}\sqrt\kappa\ge n$ and the accuracy restriction ensures $L\Delta\eps^{-2}\ge1$. Both methods are zero-respecting; Appendix~\ref{app:catalyst} records the imported accuracy guarantee and oracle implementation.
\end{proof}
Use RPDA below $\sqrt n$ and Catalyst above it; at the transition both accuracy-dependent terms equal $nL\Delta\eps^{-2}$. RPDA's logarithms arise only from initialization and refinement. Catalyst retains the polylogarithmic factors of \citet[Appendix D.4]{zhang}. The lower bound and RPDA guarantee follow.

\begin{theorem}[Averaged smoothness]\label{thm:as}
There are universal constants $c_0,c_1>0$ such that, for $n\ge2$, $\kappa=L/\mu\ge1$, and $0<\eps^2\le c_0L\Delta$,
\begin{equation}\label{eq:mainlower}
 \N_\eps\bigl(\Fas(n,L,\mu,\Delta)\bigr)
 \ge c_1\left[n+\min\{\sqrt n\,\kappa,n^{3/4}\sqrt\kappa\}\,L\Delta\eps^{-2}\right].
\end{equation}
Every component of the hard instances is $\mu$-strongly concave in $y$, with a quadratic dual term and at most bilinear coupling. The same order holds for worst-case expected query complexity, for $\E\norm{\nabla\Phi(\hat x)}\le\eps$, and for success probability at least $2/3$ under a fixed budget.
\end{theorem}
The first branch of \eqref{eq:mainlower} is matched by RPDA. All lower bounds concern Definition~\ref{def:zr}.

\begin{theorem}[RPDA upper bound]\label{thm:smallupper}
For $n\ge2$, $1\le\kappa\le\sqrt n$, and $\eps>0$, Algorithm~\ref{alg:recursiveprox} is zero-respecting and attains \eqref{eq:criterion} on $\Fas(n,L,\mu,\Delta)$ with a fixed query budget of
\begin{equation}\label{eq:smallupper}
 \tO\!\left(n+\sqrt n\,\kappa L\Delta\eps^{-2}\right).
\end{equation}
For $\kappa>\sqrt n$, RPDA still attains \eqref{eq:criterion}, with fixed query budget $\tO((n+\kappa^2)(1+L\Delta\eps^{-2}))$. Only the average function must be strongly concave in $y$; no component proximal oracle is used.
\end{theorem}
\begin{theorem}[Individual smoothness]\label{thm:is}
There are universal constants $c_0,c_2>0$ such that, for $n\ge2$, $\kappa\ge1$, and $0<\eps^2\le c_0L\Delta$,
\begin{equation}\label{eq:islower}
 \N_\eps\bigl(\Bis(n,L,\mu,\Delta)\bigr)
 \ge c_2\left(n+\min\{\kappa,\sqrt{n\kappa}\}\,L\Delta\eps^{-2}\right).
\end{equation}
The dual terms are quadratic. For $\kappa\ge2$, the dual-path instances have a common separable primal term; for $1\le\kappa<2$, a common minimization chain suffices. The fixed-budget and expected-query conclusions of Theorem~\ref{thm:as} both apply.
\end{theorem}

\begin{proposition}[Bilinear upper bound from existing methods]\label{prop:upper}
For $0<\eps^2\le L\Delta$, an inexact proximal-point method using the separable finite-sum saddle-point solver of \citet{jst} attains \eqref{eq:criterion} on $\Bis(n,L,\mu,\Delta)$ with
\begin{equation}\label{eq:isupper}
 \tO\!\left((n+\sqrt{n\kappa})L\Delta\eps^{-2}\right)
\end{equation}
expected IFO queries.
\end{proposition}
For $2\le\kappa<n$, reducing the dual-path amplitude to $h=\Theta(L\kappa/n)$ gives the lower bound; Appendix~\ref{app:issmallkappa} covers $1\le\kappa<2$. The minimization lower bound of \citet{densehiding} supplies a $\sqrt n$ coefficient for all $\kappa$, while RPDA improves the upper bound below $\sqrt n$. Table~\ref{tab:comparison-is} combines these results; Proposition~\ref{prop:iscomparison} proves the comparison.

The remaining coefficient gap is, up to constants and logarithms, $\kappa$ for $1\le\kappa\le\sqrt n$, $n/\kappa$ for $\sqrt n\le\kappa\le n$, and one for $\kappa\ge n$. Its maximum is $\sqrt n$; additive $n$ terms can reduce the ratio of the full bounds. Proposition~\ref{prop:iscap} limits our diagonal-cancelled dual-path family, including nonuniform weights and multiple owners, without ruling out other IS constructions.

At $\kappa=1$, the class reduces to finite-sum minimization (Proposition~\ref{prop:boundary}). Its rate is $\Theta(n+\sqrt n\,L\Delta\eps^{-2})$ in the stated accuracy range, so an $nL\Delta\eps^{-2}$ lower bound cannot hold uniformly in $\kappa$.

\section{Construction and Proof Overview}\label{sec:construction}
We describe the averaged-smooth instance for $\kappa\ge2$. It combines a primal--dual path with a random assignment of internal dual edges. Appendix~\ref{app:smallkappa} handles $1\le\kappa<2$ by a colored primal chain.

\subsection{Primal chain and dual paths}
Let $x\in\R^{d+1}$ and let $\Gamma$ be the scalar potential in \eqref{eq:gamma}. Define
\begin{equation}\label{eq:basechain}
 F_d(x)=\frac12(x_1-1)^2+
 \frac12\sum_{j=1}^{d}(x_j-x_{j+1})^2+
 \sum_{j=1}^{d}\Gamma(x_j).
\end{equation}
The chain in \eqref{eq:basechain} is the parameter-one specialization of \citet[Equation (9)]{cdhs}. It has initial gap at most $1/2+10d$, and its gradient has norm greater than $1/4$ whenever $x_d=x_{d+1}=0$; see Appendix~\ref{app:chain}.

For each link $(x_j,x_{j+1})$, introduce a dual block $y^{(j)}\in\R^m$. Set
\begin{equation}\label{eq:parameters}
 w=\frac{L}{16\sqrt n},\qquad
 m=\max\left\{2,\left\lfloor\sqrt{w/\mu}\right\rfloor\right\},
 \qquad h=\frac{L}{2000}\min\left\{1,\frac\kappa{\sqrt n}\right\},
\end{equation}
and let
\begin{equation}\label{eq:dualpath}
 P_m=\sum_{k=1}^{m-1}(e_k-e_{k+1})(e_k-e_{k+1})^\top,
 \qquad H=\mu I+wP_m.
\end{equation}
The coupling uses only the first and last coordinate of each dual block. A separable quadratic correction cancels the extra diagonal term created by dual maximization. Appendix~\ref{app:eliminate} specifies the resulting primal term $g$ and coupling matrix $B$ and verifies that the eliminated objective is exactly $hF_d$. Appendix~\ref{app:chain} also proves the endpoint estimates needed for smoothness.

\subsection{Random assignment to components}
Independently for each block $j$, draw its first edge owner $\sigma_{j,1}$ uniformly from $[n]$. For $k\ge2$, draw $\sigma_{j,k}$ uniformly from $[n]\setminus\{\sigma_{j,k-1}\}$, using independent randomness conditional on the preceding owner. Thus equal-owner edges have disjoint supports.

Let $v_{j,k}$ denote the vector $e_k-e_{k+1}$ embedded in dual block $j$, and define
\begin{equation}\label{eq:instance}
 \begin{gathered}
 P_i=\sum_{(j,k):\,\sigma_{j,k}=i}v_{j,k}v_{j,k}^\top,
 \qquad K_i=\mu I+nwP_i,\\
 f_i(x,y)=g(x)+\langle Bx,y\rangle-\tfrac12y^\top K_i y.
 \end{gathered}
\end{equation}
The average dual Hessian is $K=I_d\otimes H$. Equations~\eqref{eq:instance} and~\eqref{eq:bridge} give
\begin{equation}\label{eq:primal}
 \Phi(x)=hF_d(x),\qquad y^*(0)=0.
\end{equation}
Every $K_i\succeq\mu I$. Disjoint equal-owner edges give $P_i^2=2P_i$, which controls the second moment of the component Hessians. Appendix~\ref{app:smooth} combines this identity with the endpoint estimates to verify global smoothness, including the primal--dual cross terms.

\paragraph{Sequential discovery and stationarity.}
Coordinates are exposed sequentially: $x_j$ precedes $y^{(j)}_1,\ldots,y^{(j)}_m$, which precede $x_{j+1}$. Reaching $x_d$ requires traversing $J=(d-1)(m-1)$ internal edges. Each frontier owner is conditionally uniform among $n$ or $n-1$ candidates, so adaptive queries still incur a constant-probability delay of order $nJ$ (Lemma~\ref{lem:search}). In the chain's accuracy range $\eps^2\le c_0L\Delta\min\{1,\kappa/\sqrt n\}$, scaling by $s=16\eps/h$ and taking $d=\Theta(h\Delta\eps^{-2})$ gives gradient norm greater than $4\eps$ before discovery. The resulting cost is
\begin{equation}\label{eq:product}
 \Omega\bigl(n(d-1)(m-1)\bigr)
 =\Omega\!\left(\min\{\sqrt n\,\kappa,n^{3/4}\sqrt\kappa\}\,L\Delta\eps^{-2}\right).
\end{equation}
Appendix~\ref{app:scaling} proves the fixed-budget and stopping-time bounds. A hidden-component quadratic supplies $\Omega(n)$ and extends the accuracy range to $\eps^2\le c_0L\Delta$ (Appendix~\ref{app:additive}).

\paragraph{Smoothness budgets.}
Averaged smoothness allows $w=\Theta(L/\sqrt n)$, whereas individual smoothness requires $w=\Theta(L/n)$ in this construction. The separable correction permits amplitudes $h=\Theta(L\min\{1,\kappa/\sqrt n\})$ and $h=\Theta(L\min\{1,\kappa/n\})$, respectively. These choices give the two lower bounds; Appendices~\ref{app:smooth} and~\ref{app:is} verify them.

\section{Recursive Proximal Descent--Ascent (RPDA)}\label{sec:smallupper}
Algorithm~\ref{alg:recursiveprox} gives the method behind Theorem~\ref{thm:smallupper}. Its three phases are initialization, proximal iterations, and final refinement; every update is written explicitly below.

\begin{algorithm}[!tb]
\caption{RPDA: Recursive Proximal Descent--Ascent}\label{alg:recursiveprox}
\small
\begin{algorithmic}[1]
\Require $n,L,\mu,\Delta,\eps$; parameters $q,b_0,b,M,T,R_0,R_1$ from \eqref{eq:algorithmparameters}--\eqref{eq:algorithmrefinement}.
\State $z=(x,y)\gets(0,0)$; draw $J$ uniformly from $[T]$.
\For{$s=1,\ldots,R_0+T+R_1$}
  \If{$s\le R_0$} \Comment{initialization}
    \State $a\gets0$, $b_s\gets b_0$, $\mathsf{reset}\gets\mathsf{true}$.
  \ElsIf{$s\le R_0+T$} \Comment{proximal iterations}
    \State $t\gets s-R_0$, $a\gets x$, $b_s\gets b$, $\mathsf{reset}\gets[t\equiv1\pmod q]$.
    \State \textbf{if} $t=J$ \textbf{then} store $(a_*,z_*)\gets(a,z)$.
  \Else \Comment{final refinement}
    \State \textbf{if} $s=R_0+T+1$ \textbf{then} $(a,z)\gets(a_*,z_*)$.
    \State $b_s\gets b_0$, $\mathsf{reset}\gets\mathsf{true}$.
  \EndIf
  \State \textbf{if} $\mathsf{reset}$ \textbf{then} $v=(v_x,v_y)\gets n^{-1}\sum_{i=1}^n\nabla f_i(z)$.
  \For{$k=1,\ldots,7M+6$}
    \If{$k\equiv0\pmod7$} \Comment{one dual step after six primal steps}
      \State $z^+\gets(x,\ y+v_y/(4L))$.
    \Else
      \State $z^+\gets(x-[v_x+2L(x-a)]/(3L),\ y)$.
    \EndIf
    \State Draw $b_s$ independent uniform indices $I$ from $[n]$, with replacement.
    \State $v\gets v+b_s^{-1}\sum_{i\in I}[\nabla f_i(z^+)-\nabla f_i(z)]$.
    \State $z=(x,y)\gets z^+$.
  \EndFor
\EndFor
\State \Return $x$.
\end{algorithmic}
\end{algorithm}

RPDA combines proximal regularization with the recursive estimator of \citet{spider}. The upper bound is proved here by controlling this estimator across subproblems. At center $a$, set $q_a(x,y)=f(x,y)+L\norm{x-a}^2$. Its primal slices are $L$-strongly convex and $3L$-smooth. If $p_a(y)=\arg\min_xq_a(x,y)$, then $p_a$ is $1$-Lipschitz, and $D_a(y)=q_a(p_a(y),y)$ is $\mu$-strongly concave and $2L$-smooth. The minimizer $p_a$ is used only in the analysis: the algorithm takes six primal gradient steps for each dual step.

\paragraph{Estimator reuse.}
A full gradient initializes the estimator at the start of each epoch. The recursive difference update then follows every primal and dual move. Changing the proximal center affects only the known term $2L(x-a)$, so the estimator for $\nabla f$ can be retained. The main phase uses
\[
 q=\max\{1,\lfloor\sqrt n/\kappa\rfloor\},\qquad
 b=\Theta(\kappa q),\qquad M=\Theta(\kappa).
\]
All counts are deterministic. In the initialization phase the center remains zero. In the final phase it remains at the saved center $a_*$, starting from the saved pair $z_*$. These two phases reset at each subproblem and use batches $b_0=\Theta(\kappa)$.

\paragraph{Why the accumulated error is controlled.}
Let $W$ be the input dual gap plus $L$ times the squared distance to $p_a(y)$, and let $e$ be the largest estimator error during one subproblem. Lemma~\ref{lem:upperinner} proves
\begin{equation}\label{eq:upperoverview}
 P_{\mathrm{out}}\le\sigma W+C\kappa e^2/L,\qquad
 \sum_k\norm{z_{k+1}-z_k}^2\le C(W/L+\kappa e^2/L^2),
\end{equation}
where $\sigma$ is a small universal constant. Within an epoch, the estimator error is a martingale. Its maximal second moment is at most $4L^2/b$ times the expected total squared query movement. Substituting the second bound in \eqref{eq:upperoverview} and choosing $b$ a sufficiently large constant times $\kappa q$ allows that error term to be absorbed. A warm-start bound controls the change of potential when the center moves. Thus the errors remain summable across the main iterations. The envelope analysis needs only a constant relative contraction of the potential: no $1/\kappa$ subproblem tolerance or full-gradient stopping test is imposed. The dual-error factor $\kappa$ is paid through $b=\Theta(\kappa q)$, without an additional factor $\kappa$ in the inner iteration count.

\paragraph{Stationarity and query count.}
The outer analysis uses the envelope $H(a)=\min_x\{\Phi(x)+L\norm{x-a}^2\}$, whose gradient is $2L$-Lipschitz independently of $\kappa$. After initialization, the potential estimates imply
\[
 \sum_{t=1}^{T}\E\norm{\nabla H(x_t)}^2\le52L\Delta.
\]
Taking $T=\max\{1,\lceil256L\Delta/\eps^2\rceil\}$ makes a uniformly selected center nearly stationary for $H$. The final refinement computes an approximate proximal point of that center and ensures $\E\norm{\nabla\Phi(\hat x)}^2\le\eps^2$. This step is included in the query bound; no exact primal or dual solution oracle is used.

There are $O(\kappa)$ estimator updates per subproblem and one full-gradient computation per $q$ subproblems. The main phase therefore costs $O(n+T(n/q+\kappa^2q))$. For $\kappa\le\sqrt n$, initialization and refinement together cost $\tO(n+\kappa^2)=\tO(n)$, giving Theorem~\ref{thm:smallupper}. For $\kappa>\sqrt n$, the same algorithm uses $q=1$ and remains valid, but its rate is $\tO((n+\kappa^2)(1+L\Delta\eps^{-2}))$. Catalyst gives the sharper condition-number dependence in the matching accuracy range. Thus Theorem~\ref{thm:matching} uses RPDA below the threshold and Catalyst above it. Appendix~\ref{app:smallupper} supplies all constants and the complete proof.

\section{Numerical Experiments}\label{sec:experiments}
We illustrate Tables~\ref{tab:comparison} and~\ref{tab:comparison-is} on synthetic families with $d_x=d_y=10$, $n=10{,}000$, and $\kappa\in\{\sqrt n/5,\sqrt n,5\sqrt n\}=\{20,100,500\}$. The averaged-smooth family has components
\begin{equation}\label{eq:expfamily}
 f_i(x,y)=\tfrac12x^\top Q_ix+\beta\sum_{k=1}^{d_x}\psi(x_k)+\langle A_ix,y\rangle-\tfrac12y^\top M_iy,
 \qquad \psi(t)=\frac{t^2}{1+t^2},
\end{equation}
with $\beta=0.8$ and centered random perturbations of $Q_i,A_i,M_i$ whose sums vanish, so $f$ is nonconvex in $x$ and $\mu$-strongly concave in $y$; the methods receive the averaged smoothness constant $L_{\mathrm{ms}}=\norm{n^{-1}\sum_iH_i^2}^{1/2}+2\beta$ of Assumption~\ref{ass:as}, with $\mu$ calibrated to the prescribed $\kappa$. The family in $\Bis$ keeps the same primal term and uses $\norm{A_i}\le L$ and $h_i(y)=\tfrac12y^\top M_iy$ with the eigenvalues of every $M_i$ in $[\mu,L]$ and $\lambda_{\min}(n^{-1}\sum_iM_i)=\mu$; the methods receive the individual smoothness constant $L$. All methods use the constants of their analyses with order-one factors set to one and no tuning: RPDA is Algorithm~\ref{alg:recursiveprox} with $M=\lceil\kappa\rceil$, $q=\max\{1,\lfloor\sqrt n/\kappa\rfloor\}$, $b_0=b=\lceil\kappa q\rceil$, and $R_0=2$; the SREDA-type method is SPIDER-GDA with batch $\sqrt n$ and $\kappa$ dual steps per primal step \citep{sreda,zhang}; the Catalyst-type method solves the same proximal subproblem as RPDA with a fresh full gradient at every center \citep{zhang}; the proximal reduction of Proposition~\ref{prop:upper} uses Algorithms~6 and~7 of \citet{jst} as the inner solver. We start from $\norm{x_0}=2$, $y_0=y^*(x_0)$, stop at $\norm{\nabla\Phi(x)}\le5\times10^{-3}$, and report the running minimum of the envelope gradient $\nabla H(x)=2L(x-p(x))$ over three seeds.

Figure~\ref{fig:experiments} reproduces the two regimes of \eqref{eq:matchedcomplexity}: RPDA is fastest at $\kappa=\sqrt n/5$, ties the Catalyst-type method at $\kappa=\sqrt n$, where the two accuracy-dependent terms coincide, and is slower above it, while the $\kappa^2$ dependence of the SREDA-type method appears as its near-horizontal curves for $\kappa\ge\sqrt n$. In Figure~\ref{fig:is-experiments}, RPDA is faster for $\kappa\le\sqrt n$ and the proximal reduction, whose cost is nearly independent of $\kappa$, is faster at $\kappa=5\sqrt n$, so the method with the better bound in Table~\ref{tab:comparison-is} is the faster one on each side of the threshold. These runs illustrate the predicted orderings on one family and do not establish worst-case rates.

\begin{figure}[!tbp]
\centering
\includegraphics[width=0.9\textwidth]{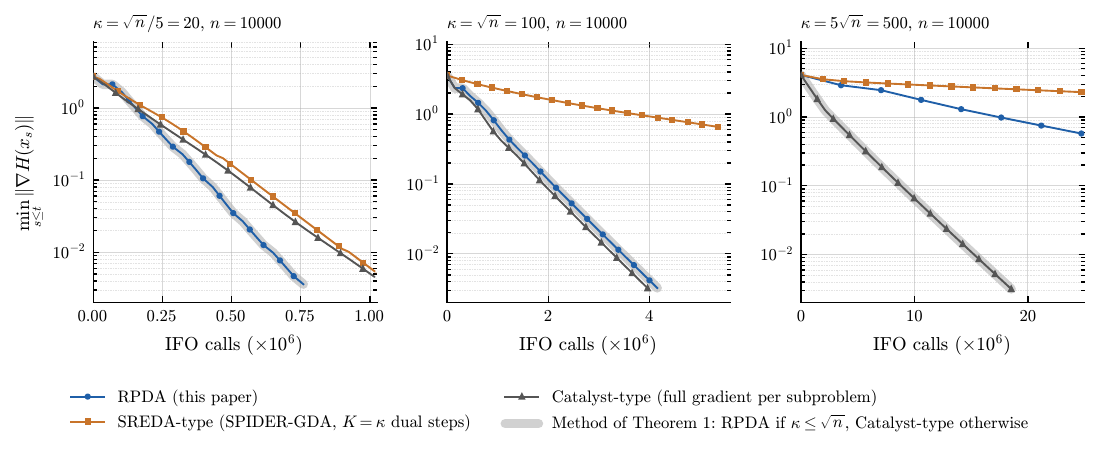}
\caption{Convergence under averaged smoothness (Table~\ref{tab:comparison}): running minimum of the envelope gradient against IFO calls; lines show the seed with median finishing cost and bands span three seeds. The light band marks the method of Theorem~\ref{thm:matching}: RPDA for $\kappa\le\sqrt n$ and the Catalyst-type method for $\kappa>\sqrt n$.}
\label{fig:experiments}
\end{figure}

\begin{figure}[!tbp]
\centering
\includegraphics[width=0.9\textwidth]{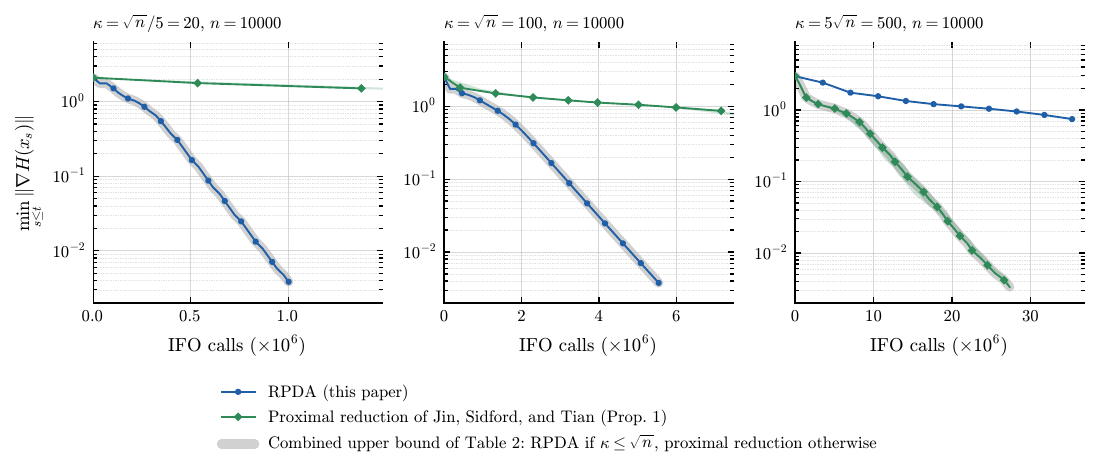}
\caption{Convergence on the bilinear individually smooth subclass (Table~\ref{tab:comparison-is}), with the conventions of Figure~\ref{fig:experiments}. The light band marks the combined upper bound of Table~\ref{tab:comparison-is}: RPDA for $\kappa\le\sqrt n$ and the proximal reduction of Proposition~\ref{prop:upper} for $\kappa>\sqrt n$.}
\label{fig:is-experiments}
\end{figure}

\section{Conclusion}\label{sec:conclusion}
Our main result is the $\tTheta$ characterization \eqref{eq:matchedcomplexity} of averaged-smooth minimax expected query complexity throughout $\kappa\ge1$, in the specified accuracy range and zero-respecting model. Dual paths give the lower bound, and recursive estimation across proximal subproblems closes the small-condition-number gap. Under individual smoothness, Table~\ref{tab:comparison-is} extends the comparison to all $\kappa$ but leaves a polynomial gap at intermediate condition numbers. Our dual-path limitation identifies a restriction of this construction, not of the IS--B class. Unrestricted NC--SC queries, component proximal oracles, and general nonlinear coupling remain open.

\section*{Acknowledgements}
The author thanks Lesi Chen of Tsinghua University for helpful discussions of this open problem and valuable feedback on the initial draft of this manuscript.

Codex with GPT-6 Astra and Claude Science Opus 5.5 were used to assist with the development of the construction and proofs, numerical checks of algebraic identities, the review of related literature, and the preparation and revision of this manuscript. The author is responsible for the mathematical claims, references, and final presentation.

\begingroup
\fontsize{10}{11}\selectfont 
\bibliographystyle{unsrtnat}
\bibliography{references}
\endgroup

\clearpage
\appendix
\linespread{0.98}\selectfont 
\section{Properties of the Primal and Dual Chains}\label{app:chain}
\subsection{The nonconvex primal chain}
The scalar potential used in \eqref{eq:basechain} is
\begin{equation}\label{eq:gamma}
 \Gamma(t)=120\int_1^t\frac{u^2(u-1)}{1+u^2}\,du.
\end{equation}
For the function $F_d$ in \eqref{eq:basechain}, the specialization of \citet[Lemmas 2--4]{cdhs} used in this paper is
\begin{equation}\label{eq:base}
 \begin{gathered}
 F_d\ge0,\qquad F_d(\mathbf1)=0,\qquad F_d(0)\le\tfrac12+10d,\\
 x_d=x_{d+1}=0\quad\Longrightarrow\quad\norm{\nabla F_d(x)}>\tfrac14.
 \end{gathered}
\end{equation}
Precisely, \citet[Lemma 3]{cdhs} state that, for $T\ge1$, $r\ge1$, and $0<\nu\le1$, their function $\bar f_{T,\nu,r}$ in Equation~(9) satisfies $\norm{\nabla\bar f_{T,\nu,r}(x)}>\nu^{3/4}/4$ whenever $x_T=x_{T+1}=0$. We have $F_d=\bar f_{d,1,1}$ and $\Gamma=\Upsilon_1$. Here $\nu$ denotes the parameter called $\mu$ in that paper, not our dual strong-concavity parameter. The hypothesis is exactly the two zero coordinates used in \eqref{eq:base}; no bound on earlier coordinates is required. The last implication is the only imported non-elementary chain estimate.

The remaining elementary properties needed below follow from
\[
 \Gamma'(t)=120\frac{t^2(t-1)}{1+t^2},\qquad
 \Gamma''(t)=120\left[1-\frac{1-t^2+2t}{(1+t^2)^2}\right].
\]
Thus $\Gamma'(0)=\Gamma'(1)=0$. The sign of $\Gamma'$ shows that $\Gamma(t)\ge\Gamma(1)=0$, and
\[
 \Gamma(0)=120\int_0^1\frac{u^2(1-u)}{1+u^2}\,du
 \le120\int_0^1u^2(1-u)\,du=10.
\]
Both $(1-t^2)/(1+t^2)^2$ and $2t/(1+t^2)^2$ have absolute value at most one. Hence $|\Gamma''(t)|\le360$ globally. We use this convenient bound without optimizing it.

\subsection{Eliminating a dual block}\label{app:eliminate}
The free-endpoint Laplacian $P_m$ is the dual-path structure used by \citet{li}. Define
\begin{equation}\label{eq:endpoints}
 a=(H^{-1})_{11}=(H^{-1})_{mm},\qquad
 b=(H^{-1})_{1m}>0,\qquad
 \rho=a/b,\qquad c=\sqrt{h/b}.
\end{equation}
Lemma~\ref{lem:endpoint} bounds $h(\rho-1)$ and $h/b$ uniformly in $n$ and $\kappa$. Maximizing one block gives
\begin{equation}\label{eq:bridge}
 \begin{aligned}
 \max_y\left\{c(uy_1-vy_m)-\tfrac12y^\top Hy\right\}
 &=\frac{c^2}{2}\bigl(a(u^2+v^2)-2buv\bigr)\\
 &=\frac h2(u-v)^2+\frac{h(\rho-1)}2(u^2+v^2).
 \end{aligned}
\end{equation}
The additional diagonal term is separable. We cancel it by setting
\begin{equation}\label{eq:g}
 g(x)=h\left[\frac12(x_1-1)^2+\sum_{j=1}^{d}\Gamma(x_j)\right]
 -\frac{h(\rho-1)}2\sum_{j=1}^{d}(x_j^2+x_{j+1}^2).
\end{equation}
Let $B$ be the linear map satisfying
\begin{equation}\label{eq:B}
 \langle Bx,y\rangle
 =c\sum_{j=1}^{d}\bigl(x_jy^{(j)}_1-x_{j+1}y^{(j)}_m\bigr).
\end{equation}
Because $m\ge2$, each primal variable appears at at most two distinct endpoints and $\norm B\le\sqrt2c$.

\subsection{Endpoint estimates for the dual path}
\begin{lemma}\label{lem:endpoint}
For the parameters in \eqref{eq:parameters} and endpoints in \eqref{eq:endpoints}, with $n\ge2$, $\kappa\ge2$, and $\vartheta:=\min\{1,\kappa/\sqrt n\}$, we have $\rho\ge1$ and
\begin{equation}\label{eq:endpointbounds}
 \vartheta(\rho-1)\le16,\qquad \vartheta b^{-1}\le9L,\qquad
 m-1\ge\frac1{12}\sqrt{\frac\kappa{\sqrt n}}.
\end{equation}
\end{lemma}
\begin{proof}
Set $r=\mu/w$ and $t=\operatorname{arcosh}(1+r/2)$. If $u=H^{-1}e_m$, the first boundary equation and the interior recurrence are
\[
 u_2=(1+r)u_1,\qquad u_{j+1}=(2+r)u_j-u_{j-1}\quad(2\le j<m).
\]
Consequently,
\begin{equation}\label{eq:cosh}
 \frac{u_j}{u_1}=
 \frac{\cosh((j-\tfrac12)t)}{\cosh(t/2)},\qquad
 \rho=\frac{\cosh((m-\tfrac12)t)}{\cosh(t/2)}.
\end{equation}
Since $\mathbf1^\top H=\mu\mathbf1^\top$, we have $\sum_j u_j=1/\mu$. Formula~\eqref{eq:cosh} then gives $b=u_1>0$ and $b\le u_j\le a$.

If $m\ge3$, then $m\le\sqrt{w/\mu}$. The inequality $t\le\sqrt r$ follows from $\cosh(\sqrt r)\ge1+r/2$. Hence $\rho\le\cosh(1)<2$, and
\[
 \frac1\mu=\sum_j u_j\le2mb,
 \qquad b^{-1}\le2\mu m\le2\sqrt{\mu w}\le2L.
\]
If $m=2$, direct inversion gives
\[
 \rho=1+r,\qquad b^{-1}=\mu(r+2).
\]
Here $\vartheta r\le16$ and $\mu\le L/2$, so $\vartheta(\rho-1)\le16$ and $\vartheta b^{-1}=\mu(\vartheta r+2\vartheta)\le9L$. For $m\ge3$, the already proved bounds $\rho<2$, $b^{-1}\le2L$, and $\vartheta\le1$ imply the same estimates.

For any $q>0$, $\max\{2,\lfloor q\rfloor\}-1\ge q/3$. Indeed, when $q<3$ the left side is at least one; when $q\ge3$ it is at least $q-2\ge q/3$. Taking $q=\sqrt{w/\mu}=\tfrac14\sqrt{\kappa/\sqrt n}$ proves the last bound.
\end{proof}

\section{Smoothness and Strong Concavity}\label{app:smooth}
We verify the function-class conditions globally. The matrices $P_i$ in \eqref{eq:instance} sum disjoint edge outer products within each component. With $P=\sum_iP_i=I_d\otimes P_m$, this gives
\begin{equation}\label{eq:matching}
 P_i^2=2P_i,\qquad \norm{P_i}\le2,\qquad \norm P\le4.
\end{equation}
In particular, every $K_i\succeq\mu I$, which proves componentwise strong concavity.

From separability of $g$ and Lemma~\ref{lem:endpoint},
\begin{equation}\label{eq:gBbounds}
 \begin{aligned}
 \Lip(\nabla g)&\le\frac{L}{2000}\{361\vartheta+2\vartheta(\rho-1)\}
 \le\frac{393L}{2000}<\frac L5,\\
 \norm B&\le\sqrt{2h/b}\le\sqrt{\frac{18}{2000}}L<\frac L{10}.
 \end{aligned}
\end{equation}
Expanding $K_i^2$ and using $P_i^2=2P_i$ gives
\begin{equation}\label{eq:secondmoment}
 \begin{aligned}
 \frac1n\sum_iK_i^2
 &=\mu^2I+2\mu w\sum_iP_i+nw^2\sum_iP_i^2\\
 &=\mu^2I+(2\mu w+2nw^2)P
 \preceq(\mu^2+8\mu w+8nw^2)I.
 \end{aligned}
\end{equation}
For $n\ge2$, $\mu\le L/2$, and $w=L/(16\sqrt n)$,
\begin{equation}\label{eq:dualms}
 \mu^2+8\mu w+8nw^2
 \le L^2\left(\frac14+\frac1{4\sqrt2}+\frac1{32}\right)
 <\frac{49L^2}{100}.
\end{equation}
Let $z-z'=(u,v)$. Apply the triangle inequality to the stacked component gradients, with norm $(n^{-1}\sum_i\norm{\cdot}_i^2)^{1/2}$:
\begin{align*}
 &\left(\frac1n\sum_i\norm{\nabla f_i(z)-\nabla f_i(z')}^2\right)^{1/2}\\
 &\quad\le\norm{\nabla g(x)-\nabla g(x')}
       +\norm{(B^\top v,Bu)}
       +\left(\frac1n\sum_i\norm{K_i v}^2\right)^{1/2}\\
 &\quad\le\left(\Lip(\nabla g)+\norm B+\frac{7L}{10}\right)\norm{z-z'}
       \le L\norm{z-z'}.
\end{align*}
This proves Assumption~\ref{ass:as}, including all coupling cross terms. Finally, averaging the component Hessians gives $K=I_d\otimes H$. Maximizing each block and applying \eqref{eq:bridge} proves \eqref{eq:primal}. Since $B0=0$, the dual maximizer at the origin is zero.

\subsection{Effect of the smoothness assumption}
The distinction between averaged and individual smoothness enters through the permissible dual edge weight $w$. The matching identity in \eqref{eq:matching} gives
\begin{equation}\label{eq:twobudgets}
 \left\|\frac1n\sum_iK_i^2\right\|
 \le\mu^2+8\mu w+8nw^2,
 \qquad
 \max_i\norm{K_i}\le\mu+2nw.
\end{equation}
Thus averaged smoothness permits $w$ of order $L/\sqrt n$. If $\kappa\ge\sqrt n$, the primal amplitude $h$ is of order $L$, giving the matching rate. Below this threshold, $m=2$ and $\rho-1=\mu/w=16\sqrt n/\kappa$. Keeping $h$ of order $L$ is not uniformly feasible as $\kappa/\sqrt n$ decreases, because of the diagonal correction in \eqref{eq:g}. Choosing $h$ of order $L\kappa/\sqrt n$ preserves smoothness and gives the coefficient $\sqrt n\,\kappa$ in \eqref{eq:product}.

For individual smoothness, take $w=L/(16n)$ and $h=(L/2000)\min\{1,\kappa/n\}$. When $\kappa\ge n$, the amplitude is constant and $m-1\ge\tfrac1{12}\sqrt{\kappa/n}$. When $2\le\kappa<n$, a single internal edge and a reduced amplitude give the coefficient $\kappa$. Appendix~\ref{app:is} verifies both branches.

\section{Adaptive Component Search}\label{app:search}
\subsection{Support and information revealed by a query}
Let the exposed set be the union of the supports of all preceding oracle responses, in the joint ordering of Section~\ref{sec:construction}. At $(0,0)$, the only nonzero gradient entry is $-h$ at $x_1$. The function $g$ is separable and has zero derivative at every other zero coordinate. The maps $B$ and $K_i$ connect only adjacent coordinates in this ordering. Induction therefore shows that the exposed set is a prefix and each query can extend it by at most one coordinate.

Suppose the next edge lies inside block $j$, between coordinates $k$ and $k+1$. The next coordinate is exposed only by a query to $i=\sigma_{j,k}$ with $y^{(j)}_k\ne0$. The edge's diagonal contribution at coordinate $k$ may also reveal whether $i$ is its owner, but this is the same membership test. Every later edge has both input coordinates zero, so its contribution to the response vanishes. In particular, the query revealing coordinate $k+1$ contains no information about the owner of edge $k+1$: that edge's input coordinates were both zero when the query was made. Primal gradient responses do not depend on the coloring at all.

Owners of earlier edges can be granted to the algorithm without weakening the lower bound. Conditional on those owners and the transcript before the first test at the new frontier, the current owner is uniform on $[n]$ for the first edge of a block, or on $[n]\setminus\{\sigma_{j,k-1}\}$ thereafter. Failed tests at this frontier subsequently remove candidates. They do not change the conditional generation rule for any edge beyond it once its predecessor has been determined.

\begin{lemma}[Adaptive search cost]\label{lem:search}
Before exposing $x_d$, an algorithm satisfying Definition~\ref{def:zr} must traverse $J=(d-1)(m-1)$ internal dual edges. The number of required IFO calls stochastically dominates $S=\sum_{q=1}^{J}R_q$, where the $R_q$ are independent and uniform on $\{1,\ldots,n\}$ at the first edge of each block, or on $\{1,\ldots,n-1\}$ at a subsequent edge. Consequently,
\begin{equation}\label{eq:searchmoments}
 \E S\ge\frac{nJ}{2},\qquad
 \E S^2\le\frac43(\E S)^2,\qquad
 \Prob\{S>nJ/8\}\ge\frac{27}{64}.
\end{equation}
\end{lemma}

\subsection{Proof of Lemma~\ref{lem:search}}
\begin{proof}
Give previously traversed owners to the algorithm. Let $\mathcal F_t$ be the filtration generated by its private random seed, the first $t$ queries and responses, and these revealed owners. Enumerate the $J$ internal edges in their path order. For edge $q$, let $\tau_q$ be the first response count at which its upstream endpoint is exposed, and let $\eta_q$ be the first at which its downstream endpoint is exposed. These are stopping times, with $\eta_q\le\tau_{q+1}$. Complete a stopped algorithm by an admissible search through the remaining candidates; if a frontier is never reached, its infinite hitting time already dominates the claimed bound.

On reaching frontier $q$, its owner is uniform conditional on $\mathcal F_{\tau_q}$ among $N_q=n$ candidates at the start of a block and $N_q=n-1$ otherwise. A failed informative query excludes exactly its tested label; every untested owner gives the same response. Let $R_q$ be the rank of the successful label among distinct informative candidate tests. Adaptive selection therefore gives
\[
 \Prob(R_q=r\mid\mathcal F_{\tau_q})=\frac1{N_q},
 \qquad r=1,\ldots,N_q.
\]
The earlier ranks are $\mathcal F_{\tau_q}$-measurable, and $N_q$ is deterministic, so iterated conditioning proves that the $R_q$ are independent. Repeated or uninformative calls only increase $\eta_q-\tau_q$, giving $\eta_q-\tau_q\ge R_q$ pathwise on this coupling. The intervals are disjoint, hence exposing $x_d$ requires at least $\sum_qR_q$ calls. For unvisited frontiers, extend the coupling with independent fresh ranks. Inferring the final candidate still requires a successful oracle call to expose the next coordinate; this also covers $n=2$.

For a variable uniform on $\{1,\ldots,N\}$, the mean is $(N+1)/2$ and the variance is $(N^2-1)/12$. Thus
\[
 \E S\ge nJ/2,\qquad
 \operatorname{Var}(S)\le Jn^2/12\le(\E S)^2/(3J)\le(\E S)^2/3.
\]
Paley--Zygmund with parameter $1/4$ gives
\[
 \Prob\{S>\E S/4\}
 \ge\frac{(3/4)^2(\E S)^2}{\E S^2}\ge\frac{27}{64}.
\]
Since $\E S/4\ge nJ/8$, this proves \eqref{eq:searchmoments}.
\end{proof}

The proof does not assume that a dual block is solved to any prescribed numerical accuracy. It only counts the oracle responses needed to expose its opposite endpoint. Reusing a solution from an earlier block cannot expose coordinates in a later block, and assigning an arbitrary value to the first endpoint cannot skip the internal support sequence. This distinction is what makes the cost multiply across primal links.

\begin{remark}[Role of the support restriction]\label{rem:support}
The average function in \eqref{eq:primal} is independent of the edge assignment. Randomness changes only the component decomposition. Definition~\ref{def:zr} prevents nonzero outputs in unexposed coordinates, even if the family is known. Unrestricted randomized algorithms require a separate hiding argument.
\end{remark}

\section{Proof of Theorem~\ref{thm:as}}\label{app:scaling}
\subsection{The dual-path construction for \texorpdfstring{$\kappa\ge2$}{kappa at least 2}}
We first prove the chain bound under the temporary restriction $\eps^2\le c_0L\Delta\vartheta$, where $\vartheta=\min\{1,\kappa/\sqrt n\}$. Appendix~\ref{app:additive} removes the factor $\vartheta$ from the final accuracy range. Use \eqref{eq:parameters} and set
\begin{equation}\label{eq:scaling}
 s=\frac{16\eps}{h},\qquad
 d=\left\lfloor\frac{\Delta}{11hs^2}\right\rfloor,
 \qquad \widetilde f_i(x,y)=s^2f_i(x/s,y/s).
\end{equation}
It suffices to use the universal accuracy constant
\begin{equation}\label{eq:c0}
 c_0=\frac1{45056000}.
\end{equation}
Under $\eps^2\le c_0L\Delta\vartheta$, we have $\Delta/(11hs^2)\ge8$; only a lower bound of four is needed. This constant is not optimized. Scaling preserves all smoothness, strong-concavity, and support properties. The scaled primal objective satisfies
\[
 \widetilde\Phi(x)=hs^2F_d(x/s),\qquad
 \widetilde\Phi(0)-\inf\widetilde\Phi
 \le hs^2(\tfrac12+10d)\le11hs^2d\le\Delta.
\]
Exact dual initialization is also preserved.

If $x_d$ has not been exposed, an admissible output has $x_d=x_{d+1}=0$. Equation~\eqref{eq:base} gives
\begin{equation}\label{eq:unexposedgrad}
 \norm{\nabla\widetilde\Phi(x)}>\frac{hs}{4}=4\eps.
\end{equation}
For any query budget $Q\le nJ/8$, Lemma~\ref{lem:search} therefore implies
\[
 \E_{\sigma,\mathsf A}\norm{\nabla\widetilde\Phi(\hat x)}^2
 >16\eps^2\frac{27}{64}>\eps^2,
\]
where the expectation includes the random coloring and the algorithm's private randomness. Hence some fixed coloring violates the expected squared-gradient guarantee. The same argument gives expected gradient norm greater than $4\eps(27/64)>\eps$ and failure probability at least $27/64>1/3$.

Writing $D=\Delta/(11hs^2)\ge4$, we have $d-1\ge D/2$. When $\kappa\ge\sqrt n$, $\vartheta=1$ and Lemma~\ref{lem:endpoint} gives $n(m-1)\ge n^{3/4}\sqrt\kappa/12$. When $2\le\kappa<\sqrt n$, we have $m=2$ and $n\vartheta=\sqrt n\,\kappa$. Consequently,
\begin{align*}
 nJ&=n(d-1)(m-1)
 =\Omega\!\left(n\vartheta(m-1)L\Delta\eps^{-2}\right)\\
 &=\Omega\!\left(\min\{\sqrt n\,\kappa,n^{3/4}\sqrt\kappa\}\,L\Delta\eps^{-2}\right),
\end{align*}
which proves the fixed-budget assertion.

For the expected-query assertion, suppose an algorithm has stopping time $\tau$ with $\E_{\mathsf A}\tau\le nJ/64$ for every coloring. Averaging over colorings and applying Markov's inequality gives $\Prob\{\tau>nJ/8\}\le1/8$. Under the coupling of Lemma~\ref{lem:search},
\[
 \Prob\{\text{$x_d$ is unexposed at time $\tau$}\}
 \ge\frac{27}{64}-\frac18=\frac{19}{64}.
\]
Equation~\eqref{eq:unexposedgrad} then contradicts the expected squared-gradient guarantee on average over colorings. The expected gradient norm is likewise greater than $4\eps(19/64)>\eps$. Some fixed instance must fail, proving the same asymptotic lower bound for worst-case expected query complexity.

\subsection{The remaining range \texorpdfstring{$1\le\kappa<2$}{1 at most kappa below 2}}\label{app:smallkappa}
In this range we use a colored primal chain, with an independent quadratic dual variable. Let $h_0=L/(2000\sqrt n)$ and color the edges $(x_j,x_{j+1})$ of \eqref{eq:basechain}: the first owner is uniform on $[n]$, and each subsequent owner is uniform among the labels other than its predecessor. Define
\begin{equation}\label{eq:smallchain}
 f_i(x,y)=h_0\left[\frac12(x_1-1)^2+\sum_{j=1}^d\Gamma(x_j)\right]
 +\frac{nh_0}{2}\sum_{j:\,\sigma_j=i}(x_j-x_{j+1})^2
 -\frac\mu2y^2.
\end{equation}
The average primal function is $h_0F_d$, and $y^*(0)=0$. For the primal edge matrices $P_i$, the matching identities give $n^{-1}\sum_i(nh_0P_i)^2=2nh_0^2P\preceq8nh_0^2I$. The common separable primal term has gradient Lipschitz constant at most $361h_0$. Thus its primal component gradients satisfy averaged smoothness with constant
\[
 361h_0+\sqrt{8n}\,h_0
 =\frac{L}{2000}\left(\frac{361}{\sqrt n}+\sqrt8\right)<L.
\]
Since the dual coordinate is decoupled, the squared joint increment is the sum of the primal squared increment and $\mu^2|y-y'|^2$. The joint averaged-smoothness constant is therefore at most $L$, even at $\mu=L$.

Only $x_1$ is exposed initially. Every subsequent primal coordinate requires finding the current edge owner. The proof of Lemma~\ref{lem:search} applies with $J=d-1$: no later edge contributes until an incident coordinate is exposed. Use \eqref{eq:scaling} with $h=h_0$. Because $\vartheta\le2/\sqrt n$, the temporary restriction $\eps^2\le c_0L\Delta\vartheta$ and \eqref{eq:c0} ensure $D=\Delta h_0/(2816\eps^2)\ge4$. The same gradient barrier and fixed-budget and stopping-time arguments yield
\[
 \Omega\!\left(nh_0\Delta\eps^{-2}\right)
 =\Omega\!\left(\sqrt n\,L\Delta\eps^{-2}\right).
\]
For $1\le\kappa<2$, the coefficient $\min\{\sqrt n\,\kappa,n^{3/4}\sqrt\kappa\}$ is at most $2\sqrt n$, so this proves the accuracy-dependent term of Theorem~\ref{thm:as} in the remaining range.

\subsection{The additive component-search term}\label{app:additive}
For every $\kappa\ge1$ and $\eps^2\le L\Delta/8$, there is a separate $\Omega(n)$ lower bound, even under individual smoothness. Choose $i_*\in[n]$ uniformly, put $a=\sqrt{2L\Delta}$, and define the scalar instance
\[
 f_i(x,y)=\frac L2x^2-na\,\mathbf1_{\{i=i_*\}}x-\frac\mu2y^2.
\]
Every component is jointly $L$-smooth and $\mu$-strongly concave in $y$, and the instance belongs to $\Bis$. Its primal gap at the origin is exactly $\Delta$. Until component $i_*$ is queried, every response at the origin is zero, so a zero-respecting algorithm must keep both its queries and its output at the origin. There the primal gradient has magnitude $a\ge4\eps$.

With at most $n/4$ queries, the hidden component is missed with probability at least $3/4$, even under adaptive sampling without replacement. If instead the expected number of queries is at most $n/32$ for every hidden index, Markov's inequality bounds the probability of exceeding $n/4$ queries by $1/8$. The probability of stopping before discovery is then at least $5/8$. Both events contradict the expected squared-gradient or expected-norm guarantee; the fixed-budget statement also contradicts success probability $2/3$.

When $\eps^2\le c_0L\Delta\vartheta$, the worst case over the union of this quadratic family and the appropriate chain family is at least the maximum of their lower bounds. The inequality $\max\{u,v\}\ge(u+v)/2$ gives the additive form in \eqref{eq:mainlower}, without a direct-sum argument.

It remains to cover $c_0L\Delta\vartheta<\eps^2\le c_0L\Delta$. This interval is nonempty only if $\kappa<\sqrt n$, in which case $\vartheta=\kappa/\sqrt n$ and
\[
 \min\{\sqrt n\,\kappa,n^{3/4}\sqrt\kappa\}\,L\Delta\eps^{-2}
 =\sqrt n\,\kappa L\Delta\eps^{-2}<\frac{n}{c_0}.
\]
Thus the desired expression is at most $(1+c_0^{-1})n$. The quadratic family gives its lower bound after reducing the universal constant $c_1$. Since $c_0<1/8$, its accuracy requirement is satisfied. The argument applies equally to fixed budgets, expected query counts, expected gradient norm, and the stated success probability. This proves Theorem~\ref{thm:as} on the full range $0<\eps^2\le c_0L\Delta$.\hfill$\square$

\section{Proof of Theorem~\ref{thm:is}}\label{app:is}
For $\kappa\ge2$, use the construction of Section~\ref{sec:construction} with
\begin{equation}\label{eq:isparameters}
 w=\frac{L}{16n},\qquad
 m=\max\{2,\lfloor\sqrt{w/\mu}\rfloor\},\qquad
 h=\frac{L}{2000}\min\{1,\kappa/n\}.
\end{equation}
We first verify its two regimes, then handle $1\le\kappa<2$ with a decoupled instance. All three cases use the accuracy constant in \eqref{eq:c0}.

\subsection{The range \texorpdfstring{$\kappa\ge n$}{kappa at least n}}
Here $h=L/2000$. For $m\ge3$, the endpoint proof in Appendix~\ref{app:chain} is unchanged. For $m=2$, $r=\mu/w=16n/\kappa\le16$, so
\[
 \rho=1+r\le17,\qquad b^{-1}=\mu(r+2)\le9L.
\]
Consequently,
\[
 \Lip(\nabla g)\le393h<L/5,\qquad
 \norm B\le\sqrt{18hL}<L/10,\qquad
 \norm{K_i}\le\mu+2nw\le5L/8.
\]
For each component, the triangle inequality for the joint gradient gives smoothness at most $L/5+L/10+5L/8<L$. Each $h_i(y)=\tfrac12y^\top K_i y$ is also individually $L$-smooth and $\mu$-strongly convex, so the instance belongs to \eqref{eq:bilinearclass}.

The primal identity, coloring law, and search-rank argument are unchanged. The chain length now obeys
\[
 m-1\ge\frac13\sqrt{w/\mu}=\frac1{12}\sqrt{\kappa/n}.
\]
Apply the same scaling \eqref{eq:scaling} and accuracy constant \eqref{eq:c0}. We obtain
\[
 nJ=\Omega\!\left(\sqrt{n\kappa}\,L\Delta\eps^{-2}\right).
\]
The fixed-budget and stopping-time arguments in Appendix~\ref{app:scaling} apply unchanged. Since $\sqrt{n\kappa}\ge n$ and $L\Delta\eps^{-2}\ge1$ in this regime, the same bound includes the additive $n$ term.

\subsection{The range \texorpdfstring{$2\le\kappa<n$}{2 at most kappa below n}}
Now $m=2$ and $h=L\kappa/(2000n)$. Direct inversion gives $\rho=1+r$ and $b^{-1}=\mu(r+2)$, where $r=16n/\kappa$. Consequently,
\[
 \begin{gathered}
 h(\rho-1)=16L/2000,\qquad
 \norm{K_i}\le\mu+2nw\le5L/8,\\
 \Lip(\nabla g)\le h\{361+2(\rho-1)\}
 \le393L/2000<L/5,\\
 \norm B^2\le2h/b
 =\frac{L^2}{1000}\left(\frac{16}{\kappa}+\frac2n\right)
 \le\frac{9L^2}{1000}<\frac{L^2}{100}.
 \end{gathered}
\]
As above, every component is jointly $L$-smooth, and its quadratic dual term is $\mu$-strongly convex and $L$-smooth. Thus the instance belongs to $\Bis$. The primal identity and coloring law remain unchanged, with one internal dual edge per link.

Under the temporary restriction $\eps^2\le c_0L\Delta\kappa/n$, use \eqref{eq:scaling}. Then $D=h\Delta/(2816\eps^2)\ge8$, $d-1\ge D/2$, and $J=d-1$. Hence
\[
 nJ\ge nD/2
 =\Omega(nh\Delta\eps^{-2})
 =\Omega(\kappa L\Delta\eps^{-2}).
\]
Lemma~\ref{lem:search} and the gradient barrier \eqref{eq:unexposedgrad} yield all the fixed-budget and expected-query conclusions of Theorem~\ref{thm:as}. Combining this family with the hidden-component quadratic in Appendix~\ref{app:additive} gives the additive $n$ term.

For $c_0L\Delta\kappa/n<\eps^2\le c_0L\Delta$, we instead have $\kappa L\Delta\eps^{-2}<n/c_0$. The same hidden-component quadratic is in $\Bis$ and gives $\Omega(n+\kappa L\Delta\eps^{-2})$ after reducing the universal constant. This completes the proof for $2\le\kappa<n$.

\subsection{The range \texorpdfstring{$1\le\kappa<2$}{1 at most kappa below 2}}\label{app:issmallkappa}
The cutoff $\kappa\ge2$ above ensures $\mu\le L/2$, leaving smoothness for the coupled dual path. It is not a restriction on the theorem. For $1\le\kappa<2$, take $h_0=L/2000$ and identical components
\begin{equation}\label{eq:isdecoupledchain}
 f_i(x,y)=h_0s^2F_d(x/s)-\frac\mu2\norm y^2,
 \qquad
 s=\frac{16\eps}{h_0},\qquad
 d=\left\lfloor\frac{h_0\Delta}{2816\eps^2}\right\rfloor.
\end{equation}
The bound $\Lip(\nabla F_d)\le365$ follows from $|\Gamma''|\le360$, the anchor curvature one, and the path-Laplacian norm at most four. Since the primal and dual variables are decoupled, the joint smoothness constant is at most $\max\{365h_0,\mu\}\le L$. Thus \eqref{eq:isdecoupledchain} belongs to $\Bis$, with $A_i=0$, $y^*(0)=0$, and primal gap at most $\Delta$ by \eqref{eq:scaling}.

Under $\eps^2\le c_0L\Delta$, the quantity inside the floor in \eqref{eq:isdecoupledchain} is at least eight. Each query exposes at most one new primal coordinate, regardless of the component index. Until $x_d$ is exposed, \eqref{eq:base} gives $\norm{\nabla\Phi}>h_0s/4=4\eps$. Consequently, a fixed budget of at most $(d-1)/2$ fails both expected-gradient criteria and any positive constant success guarantee. If a stopping time satisfies $\E\tau\le(d-1)/8$, Markov's inequality gives $\Prob\{\tau\le(d-1)/2\}\ge3/4$, which again contradicts either expected-gradient criterion. This proves an $\Omega(L\Delta\eps^{-2})$ lower bound in all the required query conventions.

Combining this family with the individually smooth hidden-component quadratic of Appendix~\ref{app:additive} gives $\Omega(n+L\Delta\eps^{-2})$. Because $n\ge2$ and $1\le\kappa<2$, we have $\min\{\kappa,\sqrt{n\kappa}\}=\kappa<2$. The last bound is therefore the claimed $\Omega(n+\min\{\kappa,\sqrt{n\kappa}\}L\Delta\eps^{-2})$ after adjusting a universal constant. This completes the proof of Theorem~\ref{thm:is} for every $\kappa\ge1$.\hfill$\square$

\subsection{Combined bounds throughout the condition-number range}\label{app:iscomparison}
The next proposition records precisely what the bounds in Table~\ref{tab:comparison-is} imply. Its accuracy constant also accounts for the imported minimization lower bound and need not equal the explicit $c_0$ in \eqref{eq:c0}.

\begin{proposition}[Bilinear comparison]\label{prop:iscomparison}
There is a universal constant $c_{\mathrm{IS}}\in(0,c_0]$ such that, for $n\ge2$, $\kappa\ge1$, and $0<\eps^2\le c_{\mathrm{IS}}L\Delta$,
\begin{equation}\label{eq:iscombined}
 \begin{aligned}
 \Nexp_\eps(\Bis)
 &=\Omega\!\left(n+\max\{\sqrt n,\min\{\kappa,\sqrt{n\kappa}\}\}
 L\Delta\eps^{-2}\right),\\
 \Nexp_\eps(\Bis)
 &=\tO\!\left(n+\min\{\sqrt n\,\kappa,n+\sqrt{n\kappa}\}
 L\Delta\eps^{-2}\right),
 \end{aligned}
\end{equation}
where $\Bis=\Bis(n,L,\mu,\Delta)$. The lower bound also holds for fixed budgets. At $\kappa=1$, both fixed-budget and expected-query complexities are $\Theta(n+\sqrt n\,L\Delta\eps^{-2})$, without logarithmic factors.
\end{proposition}
\begin{proof}
Let $G=n^{-1}\sum_i g_i$ be an individually $L$-smooth minimization instance with $G(0)-\inf G\le\Delta$. Set
\[
 f_i(x,y)=g_i(x)-\frac\mu2\norm y^2.
\]
The joint smoothness constant is at most $\max\{L,\mu\}=L$, and $A_i=0$, $h_i(y)=\mu\norm y^2/2$, $y^*(x)=0$, and $\Phi=G$. The instance is in $\Bis$. Its dual gradient is known, so each joint IFO response can be simulated with one minimization query.

For all sufficiently large $n$, \citet[Theorem 3.1 and Appendix F.3, Equation (305)]{densehiding} give the lower bound $\Omega(n+\sqrt n\,L\Delta\eps^{-2})$ for unrestricted randomized IFO algorithms, including worst-case expected query budgets and expected gradient norm. Their oracle also returns function values, which only strengthens that lower bound for our gradient-only oracle. Jensen's inequality transfers it to \eqref{eq:criterion}. Choose $c_{\mathrm{IS}}$ small enough to satisfy their accuracy restriction.

For the finitely many smaller $n$, use the decoupled chain \eqref{eq:isdecoupledchain}, which is admissible whenever $\mu\le L$. Appendix~\ref{app:issmallkappa} proves its $\Omega(L\Delta\eps^{-2})$ bound, including expected budgets. The factor $\sqrt n$ is bounded by a universal constant in this finite range; Appendix~\ref{app:additive} supplies $\Omega(n)$. Thus the minimization lower bound used here holds for every $n\ge2$ after adjusting universal constants.

Taking the worst case over this family and Theorem~\ref{thm:is} proves the lower bound in \eqref{eq:iscombined} for every $\kappa\ge1$. This combination takes a maximum of two lower bounds; it does not multiply their coefficients.

For $\kappa\le\sqrt n$, apply RPDA, since $\Bis\subseteq\Fas$. Here $\sqrt n\,\kappa\le n$, so the minimum in the upper bound of \eqref{eq:iscombined} equals $\sqrt n\,\kappa$. For $\kappa\ge\sqrt n$, use Proposition~\ref{prop:upper}. In this range $\sqrt n\,\kappa$ is at least both $n$ and $\sqrt{n\kappa}$, and hence
\[
 n+\sqrt{n\kappa}
 \le2\min\{\sqrt n\,\kappa,n+\sqrt{n\kappa}\}.
\]
Since $L\Delta\eps^{-2}\ge1$, adding the one-time $n$ term does not change this upper bound. Both implementations are zero-respecting. Finally, at $\kappa=1$, Proposition~\ref{prop:boundary} gives the upper bound without logarithms, and the minimization family gives the matching lower bound.
\end{proof}

The accuracy-dependent lower coefficient is $\sqrt n$ for $1\le\kappa\le\sqrt n$, $\kappa$ for $\sqrt n\le\kappa\le n$, and $\sqrt{n\kappa}$ for $\kappa\ge n$. The corresponding upper coefficient is, up to a factor of two, $\sqrt n\,\kappa$, $n$, and $\sqrt{n\kappa}$. This proves the regime comparison in Section~\ref{sec:results}. It does not establish a $\tTheta$ characterization at intermediate condition numbers.

\subsection{A limitation of diagonal-cancelled dual paths}\label{app:iscap}
We give a limitation for the positive-Laplacian path family underlying our proof. It concerns the cost of exposing coordinates, not the complexity of solving an arbitrary IS--B instance.

\begin{proposition}[Exposure cost of a dual path]\label{prop:iscap}
Let $m\ge2$ and
\[
 H=\mu I+\sum_{k=1}^{m-1}w_k(e_k-e_{k+1})(e_k-e_{k+1})^\top,
 \qquad w_k>0.
\]
Suppose edge $k$ occurs in $r_k\in[n]$ components, with weight $nw_k/r_k$ in each owner. Assume that the component dual Hessians are positive sums of these edges and $\mu I$, and have norm at most $L$. Write $u=H^{-1}e_1$ and $\rho=u_1/u_m$. Starting with the first endpoint exposed, let $Q_{\mathrm{link}}$ be the minimum expected number of zero-respecting component queries needed to expose the last endpoint. If the primal amplitude and diagonal-cancellation budgets satisfy $h\le L$ and $h(\rho-1)\le L$, then
\begin{equation}\label{eq:iscap}
 hQ_{\mathrm{link}}=O\!\left(L\min\{\kappa,\sqrt{n\kappa}\}\right).
\end{equation}
The same conclusion holds if either budget is enlarged by a universal constant.
\end{proposition}
\begin{proof}
Positivity and the component norm bound give $w_k\le r_kL/n$. The inverse of $H$ is entrywise positive. Summing the node equations $(Hu)_j=0$ for $j=k+1,\ldots,m$ gives
\[
 w_k(u_k-u_{k+1})=\mu\sum_{j=k+1}^{m}u_j.
\]
Thus $u_1>\cdots>u_m>0$. Summing the increments and using $u_j\ge u_m$ yields
\begin{equation}\label{eq:isweightedresistance}
 \rho-1\ge\mu\sum_{k=1}^{m-1}\frac{m-k}{w_k}
 \ge\frac n\kappa\sum_{k=1}^{m-1}\frac{m-k}{r_k}.
\end{equation}
Set $S=\sum_k r_k^{-1}$. Since $0<r_k^{-1}\le1$, we have
\[
 \sum_{k=1}^{m-1}\frac{m-k}{r_k}\ge\max\{S,S^2/2\}.
\]
For the second inequality, allocate mass $S$ to the smallest integer weights first. If $S=a+t$ with $a=\lfloor S\rfloor$ and $0\le t<1$, the resulting minimum is $a(a+1)/2+(a+1)t\ge S^2/2$.

At a frontier, query independent uniform component labels, putting a nonzero value at the exposed endpoint. An owner is hit with probability $r_k/n$ on each call, so the expected cost to expose the next coordinate is at most $n/r_k$. This provides an upper bound on exposure cost, regardless of the owner assignment. Hence $Q_{\mathrm{link}}\le nS$. Including the two common primal--dual coupling edges changes this bound only by a constant factor, because they require at most two extra calls and $nS\ge m-1\ge1$.

The cancellation budget and \eqref{eq:isweightedresistance} imply
\[
 hnS\le\kappa L,
 \qquad hnS\le\frac{2\kappa L}{S}.
\]
The amplitude budget also gives $hnS\le nLS$. Therefore
\[
 hnS\le\min\{\kappa L,nLS,2\kappa L/S\}
 \le L\min\{\kappa,\sqrt{2n\kappa}\},
\]
which proves \eqref{eq:iscap}.
\end{proof}

The dual-path construction for $\kappa\ge2$ in Theorem~\ref{thm:is} satisfies these budgets and attains the coefficient in \eqref{eq:iscap}, up to universal constants. With $d=\Theta(h\Delta\eps^{-2})$ primal links, an argument charging only for exposing these paths therefore cannot improve its polynomial lower rate by adjusting positive edge weights or owner multiplicities. This statement does not cover general support-hiding constructions, dense information hiding, or hardness after support is exposed. In particular, it is not an upper bound on the minimax complexity of $\Bis$.

\section{Bilinear Upper-Bound Reduction}\label{app:upper}
We prove Proposition~\ref{prop:upper} to make the comparison in Table~\ref{tab:comparison-is} explicit. Consider \eqref{eq:bilinearclass} and, at outer iteration $t$, define
\[
 q_t(x,y)=f(x,y)+L\norm{x-x_t}^2,
 \qquad Q_t(x)=\max_yq_t(x,y)=\Phi(x)+L\norm{x-x_t}^2.
\]
Each $g_i+L\norm{\cdot-x_t}^2$ is $L$-strongly convex and $3L$-smooth. The dual term is $\mu$-strongly convex and $L$-smooth, and the bilinear coupling has norm at most $L$. The finite-sum separable saddle-point bound of \citet[Theorem 6 of the arXiv version]{jst} therefore has coefficient
\begin{equation}\label{eq:innercoefficient}
 n+\frac1{\sqrt n}\sum_{i=1}^{n}
 \left(\sqrt{3L/L}+\sqrt{L/\mu}+\frac{L}{\sqrt{L\mu}}\right)
 =O(n+\sqrt{n\kappa}).
\end{equation}
Write $\mathcal G_t(x,y)=\max_v q_t(x,v)-\min_u q_t(u,y)$ for the saddle gap, and let $\mathcal G_t^{\mathrm{in}}$ and $\mathcal G_t^{\mathrm{out}}$ be its values at the input $(x_t,y_t)$ and output $(x_{t+1},y_{t+1})$. Theorem~6 of \citet{jst} directly bounds this gap, not only distance to the saddle point. In particular, with $\chi=n+\sqrt{n\kappa}$ and any fixed $0<\theta<1$, its deterministic loop limits give a query budget
\begin{equation}\label{eq:jstrelative}
 K_\theta=O\!\left(\chi\log(2\chi)\log(2\chi/\theta)\right),
 \qquad
 \E[\mathcal G_t^{\mathrm{out}}\mid\mathcal H_t]
 \le\theta\mathcal G_t^{\mathrm{in}},
\end{equation}
where $\mathcal H_t$ is the history before the inner call. This relative form follows from their Proposition~5 and the gap conversion in the proof of Theorem~6. Equivalently, set their accuracy-to-initial-gap ratio to $\theta$: Algorithms~6--7 use that ratio only in their loop limits. These limits depend on $\theta$ and the class parameters, not on the realized input gap.

The separate gradient oracles required by this method can be implemented from the joint IFO. Cache $\nabla f_i(0,0)$ for each $i$. A query at $(x,0)$ returns $\nabla g_i(x)$ and $A_ix-\nabla h_i(0)$; a query at $(0,y)$ returns $\nabla g_i(0)+A_i^\top y$ and $-\nabla h_i(y)$. Thus the separate quantities are recovered with constant overhead after $n$ initial calls. Starting at the origin, the gradient-based implementation is zero-respecting.

Suppose each inner output satisfies $\E\mathcal G_t^{\mathrm{out}}\le\delta$. Then $\E[Q_t(x_{t+1})-\min Q_t]\le\delta$. Write $d_t=x_{t+1}-x_t$ and $e_t=\nabla Q_t(x_{t+1})$. Since $\Phi$ is $L$-weakly convex and $L(1+\kappa)$-smooth, $Q_t$ is $L$-strongly convex and $L(\kappa+3)$-smooth. Hence
\[
 \begin{gathered}
 \E\norm{e_t}^2\le2L(\kappa+3)\delta,\qquad
 \sum_{t=0}^{T-1}L\E\norm{d_t}^2\le\Delta+T\delta,\\
 \norm{\nabla\Phi(x_{t+1})}^2\le2\norm{e_t}^2+8L^2\norm{d_t}^2.
 \end{gathered}
\]
For an output uniform on $x_1,\ldots,x_T$, these inequalities imply
\begin{equation}\label{eq:proxconvert}
 \E\norm{\nabla\Phi(\hat x)}^2
 \le\frac{8L\Delta}{T}+4L(\kappa+5)\delta.
\end{equation}
Take $T=\max\{1,\lceil16L\Delta/\eps^2\rceil\}$ and $\delta=\eps^2/[16L(\kappa+5)]$.

We now choose a deterministic inner budget that ensures these unconditional gap bounds. At $t=0$, exact dual initialization gives
\[
 \mathcal G_0^{\mathrm{in}}
 \le\frac{\norm{\nabla\Phi(0)}^2}{2L}
 \le(1+\kappa)\Delta.
\]
For $t\ge1$, write $r_x=\nabla_xq_{t-1}(x_t,y_t)$ and $r_y=\nabla_yq_{t-1}(x_t,y_t)$. Smoothness of the primal and dual slices gives $\norm{r_x}^2\le6L\mathcal G_{t-1}^{\mathrm{out}}$ and $\norm{r_y}^2\le2L\mathcal G_{t-1}^{\mathrm{out}}$. The new residual is $(r_x-2Ld_{t-1},r_y)$. Applying strong convexity and strong concavity to its two slices yields
\begin{equation}\label{eq:jstwarm}
 \begin{aligned}
 \mathcal G_t^{\mathrm{in}}
 &\le\frac{\norm{r_x-2Ld_{t-1}}^2}{2L}
       +\frac{\norm{r_y}^2}{2\mu}\\
 &\le(6+\kappa)\mathcal G_{t-1}^{\mathrm{out}}
       +4L\norm{d_{t-1}}^2.
 \end{aligned}
\end{equation}
Fix, before running the algorithm,
\[
 B=(\kappa+5)\Delta+(4T+\kappa+6)\delta,
 \qquad \theta=\delta/B,
\]
using the $T$ and $\delta$ chosen above, and use the same budget $K_\theta$ at every center. The bound at $t=0$ and \eqref{eq:jstrelative} give $\E\mathcal G_0^{\mathrm{out}}\le\delta$. Inductively, if the earlier output gaps have expectation at most $\delta$, the partial descent sum gives $L\E\norm{d_{t-1}}^2\le\Delta+t\delta$. Thus
\[
 \E\mathcal G_t^{\mathrm{in}}
 \le4\Delta+(4t+6+\kappa)\delta\le B,
 \qquad
 \E\mathcal G_t^{\mathrm{out}}
 \le\theta\E\mathcal G_t^{\mathrm{in}}\le\delta.
\]
This proves the gap bounds needed in \eqref{eq:proxconvert}. The budget $K_\theta$ is determined by $n,L,\mu,\Delta,\eps$ alone; no random initial gap is used to set the iteration count.

Combining \eqref{eq:innercoefficient} and \eqref{eq:proxconvert}, and including the cached queries, gives $\tO((n+\sqrt{n\kappa})(1+L\Delta\eps^{-2}))$ expected queries. Under the assumption $\eps^2\le L\Delta$, the factor $1+L\Delta\eps^{-2}$ is at most $2L\Delta\eps^{-2}$, proving \eqref{eq:isupper}.\hfill$\square$

This reduction invokes the separable solver anew at each proximal center. Its full-sum initialization and coupling-gradient refreshes are included in \eqref{eq:innercoefficient}; warm starting controls the logarithmic dependence on the initial gap. It does not establish an amortization of the $n$ term across outer iterations. Removing that repeated cost would require a further algorithm and convergence argument, beyond this reduction.

\section{The Boundary Case and the Placement of the Additive Term}\label{app:boundary}
\begin{proposition}[The boundary case $\kappa=1$]\label{prop:boundary}
Every instance in $\Fas(n,L,L,\Delta)$ satisfies $\Phi(x)=f(x,0)$. A recursive variance-reduced method attains \eqref{eq:criterion} with a fixed budget of $O(n+\sqrt n\,L\Delta\eps^{-2})$ IFO calls.
\end{proposition}
We prove Proposition~\ref{prop:boundary}. Averaged smoothness implies that the average function $f$ has an $L$-Lipschitz joint gradient. Suppose $\mu=L$. For fixed $x$ and $v=y-y'$, strong concavity and smoothness give
\begin{align*}
 L\norm v^2
 &\le-\langle\nabla_y f(x,y)-\nabla_y f(x,y'),v\rangle\\
 &\le\norm{\nabla_y f(x,y)-\nabla_y f(x,y')}\norm v\\
 &\le\norm{\nabla f(x,y)-\nabla f(x,y')}\norm v
 \le L\norm v^2.
\end{align*}
Equality throughout implies
\[
 \nabla_y f(x,y)-\nabla_y f(x,y')=-L(y-y'),\qquad
 \nabla_x f(x,y)=\nabla_x f(x,y').
\]
Integrating in $y$ gives $f(x,y)=f(x,0)+a(x)^\top y-L\norm y^2/2$. The second identity says that $f(x,y)-f(x,0)$ is independent of $x$, so $a(x)$ is constant. Exact dual initialization makes this constant zero. Therefore
\begin{equation}\label{eq:boundaryreduction}
 f(x,y)=G(x)-\frac L2\norm y^2,\qquad
 G(x):=\frac1n\sum_i g_i(x),\qquad g_i(x):=f_i(x,0),\qquad \Phi=G.
\end{equation}
The $g_i$ satisfy averaged $L$-smoothness and $G(0)-\inf G\le\Delta$. Their gradients are obtained from the primal part of an IFO response at $(x,0)$.

For completeness, the following standard recursive-estimator argument gives a fixed query budget and the squared-gradient criterion directly. It is an epoch version of the variance-reduction principle in \citet{spider}. Take $q=b=\lceil\sqrt n\rceil$ and $\eta=1/(4L)$. At iteration $t$, compute $v_t=\nabla G(x_t)$ if $t$ is a multiple of $q$. Otherwise draw $b$ fresh, independent uniform component indices and set
\[
 v_t=v_{t-1}+\frac1b\sum_{i\in I_t}
       \bigl(\nabla g_i(x_t)-\nabla g_i(x_{t-1})\bigr).
 \qquad x_{t+1}=x_t-\eta v_t.
\]
The update for $x_{t+1}$ is used in both cases. Let $e_t=v_t-\nabla G(x_t)$ and let $t_0=q\lfloor t/q\rfloor$ denote the current epoch's first iteration. Conditional unbiasedness of each fresh increment, together with averaged smoothness, yields
\[
 \E\norm{e_t}^2\le\frac{L^2\eta^2}{b}
       \sum_{s=t_0+1}^{t}\E\norm{v_{s-1}}^2.
\]
Summing over all iterations and noting that each term appears at most $q$ times gives
\begin{equation}\label{eq:boundaryvariance}
 \sum_{t=0}^{T-1}\E\norm{e_t}^2
 \le\frac{qL^2\eta^2}{b}\sum_{t=0}^{T-1}\E\norm{v_t}^2
 =\frac1{16}\sum_{t=0}^{T-1}\E\norm{v_t}^2.
\end{equation}
Smoothness of $G$ also gives the pointwise descent inequality
\[
 G(x_{t+1})\le G(x_t)-\frac\eta2\norm{\nabla G(x_t)}^2
 -\frac\eta2(1-L\eta)\norm{v_t}^2+\frac\eta2\norm{e_t}^2.
\]
Sum this inequality and use \eqref{eq:boundaryvariance}. Since $1-L\eta=3/4>1/16$, a uniformly selected $\hat x\in\{x_0,\ldots,x_{T-1}\}$ satisfies
\[
 \E\norm{\nabla\Phi(\hat x)}^2
 \le\frac{2\Delta}{\eta T}=\frac{8L\Delta}{T}.
\]
Choosing $T=\max\{1,\lceil8L\Delta/\eps^2\rceil\}$ proves the desired accuracy. The number of queries is deterministically at most $n\lceil T/q\rceil+2bT=O(n+\sqrt n\,L\Delta\eps^{-2})$. The iterates and output are zero-respecting, because each update is a linear combination of preceding component gradients.

Together with Theorem~\ref{thm:as}, this identifies the rate $\Theta(n+\sqrt n\,L\Delta\eps^{-2})$ at $\kappa=1$ in the theorem's accuracy range. Letting $n$ grow and taking $L\Delta\eps^{-2}\gg\sqrt n$ shows that a uniform $\Omega(nL\Delta\eps^{-2})$ lower bound is impossible. Theorem~\ref{thm:smallupper} extends the matching upper bound to $1<\kappa\le\sqrt n$.\hfill$\square$
\section{Proof of the RPDA Upper Bound}\label{app:smallupper}
We prove the complexity guarantee of Algorithm~\ref{alg:recursiveprox} stated in Theorem~\ref{thm:smallupper}. The proof keeps the recursive estimator across successive proximal subproblems. All iteration counts and batch sizes are deterministic. No function values, exact minimization oracle, or component convexity assumption is used.

The iteration and sampling parameters are
\begin{equation}\label{eq:algorithmparameters}
 \begin{aligned}
 \kappa&=L/\mu,&q&=\max\{1,\lfloor\sqrt n/\kappa\rfloor\},&M&=\lceil4\kappa\log2048\rceil,\\
 b_0&=\lceil\beta\kappa\rceil,&b&=\lceil\beta\kappa q\rceil,&T&=\max\{1,\lceil256L\Delta/\eps^2\rceil\}.
 \end{aligned}
\end{equation}
The refinement counts are
\begin{equation}\label{eq:algorithmrefinement}
 R_0=\lceil\log_2(384(1+\kappa))\rceil,\qquad
 R_1=\max\left\{1,\left\lceil\log_2\!\left(\frac{1600L(1+\kappa)^2\Delta}{\eps^2}\right)\right\rceil\right\}.
\end{equation}
The parameter $\beta=120000000$ is sufficient for the proof; it is not optimized. In particular, $R_0+R_1=O(\log(2+\kappa)+\log(2+L(1+\kappa)^2\Delta\eps^{-2}))$ accounts for the logarithmic factors in the RPDA bound. At the start of final refinement, the state $z=(x,y)$ is restored to the saved pair $z_*$ and the estimator is reset there. In the main loop, the estimator is retained between centers except at epoch boundaries.

\subsection{The proximal subproblem and its dual function}
For a center $a$, define
\begin{equation}\label{eq:upperobjects}
 q_a(x,y)=f(x,y)+L\norm{x-a}^2,\quad
 p_a(y)=\arg\min_x q_a(x,y),\quad
 D_a(y)=q_a(p_a(y),y),\quad
 v_a=\max_y D_a(y).
\end{equation}
The function $q_a(\cdot,y)$ is $L$-strongly convex and $3L$-smooth. Its cross-gradient Lipschitz constant is at most $L$, so $p_a$ is $1$-Lipschitz. Consequently, $D_a$ is $\mu$-strongly concave and $2L$-smooth, with $\nabla D_a(y)=\nabla_y f(p_a(y),y)$. These statements concern the average function; averaged smoothness suffices. Let $x_a^*$ be the primal coordinate of the unique saddle point of $q_a$ and set
\begin{equation}\label{eq:upperpotential}
 G_a(y)=v_a-D_a(y),\qquad
 W_a(x,y)=G_a(y)+L\norm{x-p_a(y)}^2.
\end{equation}
Strong convexity in $x$ gives
\begin{equation}\label{eq:upperdistance}
 \norm{p_a(y)-x_a^*}^2\le\frac{2G_a(y)}L,
 \qquad \norm{x-x_a^*}^2\le\frac{4W_a(x,y)}L.
\end{equation}
Indeed, $q_a(x_a^*,y)\le v_a$, whereas $q_a(x_a^*,y)-D_a(y)\ge L\norm{x_a^*-p_a(y)}^2/2$.

\subsection{An inner iteration with bounded gradient errors}
At every queried point $z=(x,y)$, suppose an estimate $v$ of $\nabla f(z)$ is available. Add $2L(x-a)$ exactly to its primal part. An inner call first takes six primal descent steps of length $1/(3L)$ with $y$ fixed. It then repeats the following operations $M$ times: one dual ascent step of length $1/(4L)$, followed by six primal descent steps with the new $y$ fixed. Evaluate or update the estimator at every new point, including the output. Thus an inner call uses at most $7(M+1)$ new estimator evaluations.

\begin{lemma}[Inner contraction and total movement]\label{lem:upperinner}
Let $W=W_a(x_{\mathrm{in}},y_{\mathrm{in}})$, and suppose every gradient-estimation error during an inner call has norm at most $e$. Set
\[
 M=\lceil4\kappa\log2048\rceil,\qquad \sigma=1/2048.
\]
For the output potential $P=W_a(x_{\mathrm{out}},y_{\mathrm{out}})$ and the sum $V$ of squared distances between successive queried points, including the input, we have
\begin{equation}\label{eq:upperinner}
 P\le\sigma W+\frac{17\kappa e^2}{L},\qquad
 V\le C_0\left(\frac WL+\frac{Me^2}{L^2}\right),\qquad C_0=400.
\end{equation}
These inequalities are pathwise and allow arbitrary, dependent errors.
\end{lemma}
\begin{proof}
Fix $y$ and write $\psi(x)=q_a(x,y)$. Since $\psi$ is $L$-strongly convex and $3L$-smooth, $\zeta(x)=\psi(x)-L\norm x^2/2$ is convex and $2L$-smooth. Cocoercivity of $\nabla\zeta$ makes $I-\nabla\zeta/(2L)$ nonexpansive, so
\[
 \left\|u-v-\frac{\nabla\psi(u)-\nabla\psi(v)}{3L}\right\|
 =\frac23\left\|u-v-\frac{\nabla\zeta(u)-\nabla\zeta(v)}{2L}\right\|
 \le\frac23\norm{u-v}.
\]
This argument uses only differentiability and Lipschitz gradients. Taking $v=p_a(y)$ shows that an inexact primal step contracts distance by $2/3$, apart from an additive error $e/(3L)$. After six steps the distance is at most $r$ times its initial value plus $e/L$, where $r=(2/3)^6<1/8$.

Index the points after each group of six primal steps by $j=0,\ldots,M$. Write $a_j=\norm{x_j-p_a(y_j)}$, $G_j=G_a(y_j)$, $P_j=G_j+La_j^2$, and $d_j=y_{j+1}-y_j$. The gradient used for the dual update differs from $\nabla D_a(y_j)$ by a vector of norm at most $La_j+e$. Smoothness and strong concavity of $D_a$ imply
\begin{equation}\label{eq:upperonestep}
 \begin{aligned}
 G_{j+1}&\le\left(1-\frac1{4\kappa}\right)G_j
             -L\norm{d_j}^2+\frac L4a_j^2+\frac{e^2}{4L},\\
 a_{j+1}^2&\le3r^2a_j^2+3r^2\norm{d_j}^2+\frac{3e^2}{L^2}.
 \end{aligned}
\end{equation}
For clarity, the first inequality in \eqref{eq:upperonestep} follows from the smooth ascent identity with step $\eta=1/(4L)$: the decrease in the dual gap is at least $\eta\norm{\nabla D_a(y_j)}^2/2+\eta\norm{d_j/\eta}^2/4-\eta(La_j+e)^2/2$. Use $\norm{\nabla D_a(y_j)}^2\ge2\mu G_j$.
Combining the two inequalities in \eqref{eq:upperonestep} gives
\begin{equation}\label{eq:upperlyapunov}
 P_{j+1}\le\left(1-\frac1{4\kappa}\right)P_j
       -\frac L4a_j^2-\frac L2\norm{d_j}^2+\frac{4e^2}{L}.
\end{equation}
The first six primal steps give $P_0\le W+2e^2/L$. Iterating \eqref{eq:upperlyapunov} proves the first inequality in \eqref{eq:upperinner}, since the geometric error sum is at most $16\kappa e^2/L$.

Summing \eqref{eq:upperlyapunov} also gives
\[
 \sum_{j<M}a_j^2\le\frac{4P_0}L+\frac{16Me^2}{L^2},\qquad
 \sum_{j<M}\norm{d_j}^2\le\frac{2P_0}L+\frac{8Me^2}{L^2}.
\]
A group of six primal steps starting at distance $a$ from its minimizer has total squared movement at most $4a^2+22e^2/L^2$: sum the bounds $2(2/3)^{2j}a^2+32e^2/(9L^2)$ for its six steps. After a dual step, the next group's starting distance is at most $a_j+\norm{d_j}$. Including the dual steps, the total is therefore at most
\[
 4\norm{x_{\mathrm{in}}-p_a(y_{\mathrm{in}})}^2
 +8\sum_{j<M}a_j^2+9\sum_{j<M}\norm{d_j}^2
 +\frac{22(M+1)e^2}{L^2}.
\]
Substituting the preceding bounds proves the second inequality with $C_0=400$.
\end{proof}

\subsection{Changing the proximal center}
The following estimate prevents a condition-number loss when one subproblem is warm-started from the preceding output.
\begin{lemma}\label{lem:upperwarm}
For any centers $a,a'$ and any $(x,y)$,
\begin{equation}\label{eq:upperwarm}
 W_{a'}(x,y)\le2W_a(x,y)+12L\norm{a'-a}^2.
\end{equation}
\end{lemma}
\begin{proof}
Put $d=a'-a$ and $x^*=x_a^*$. The difference $q_{a'}-q_a$ is $-2L\langle d,x-x^*\rangle+c$, with $c$ independent of $(x,y)$. Testing the new primal objective at $x^*$ gives $v_{a'}\le v_a+c$. Strong convexity of $q_a(\cdot,y)$ gives
\[
 D_{a'}(y)\ge D_a(y)+c-2L\langle d,p_a(y)-x^*\rangle-2L\norm d^2.
\]
By \eqref{eq:upperdistance},
\[
 G_{a'}(y)\le\bigl(\sqrt{G_a(y)}+\sqrt{2L}\norm d\bigr)^2
 \le2G_a(y)+4L\norm d^2.
\]
The primal minimizers satisfy $\norm{p_{a'}(y)-p_a(y)}\le2\norm d$, because the change in their primal gradients is $-2Ld$. Thus $L\norm{x-p_{a'}(y)}^2\le2L\norm{x-p_a(y)}^2+8L\norm d^2$. Adding proves the claim.
\end{proof}

\subsection{Recursive estimation across an epoch}
Divide the outer proximal iterations into epochs of at most $q$ subproblems. At the first query of each epoch, compute the full joint gradient. At every subsequent new point $z_k$, draw $b$ independent uniform indices, with replacement, and set
\begin{equation}\label{eq:upperrecursive}
 v_k=v_{k-1}+\frac1b\sum_{i\in I_k}
       \bigl(\nabla f_i(z_k)-\nabla f_i(z_{k-1})\bigr).
\end{equation}
The known proximal gradient $2L(x-a)$ is added separately, so changes of center do not require a reset. The next subproblem starts at the preceding output, without a jump in the query point.

Let $e_{\mathcal E}$ be the maximum norm of $v_k-\nabla f(z_k)$ in epoch $\mathcal E$, and let $V_{\mathcal E}$ be its total squared query movement. Conditional mean zero of the fresh increments makes the errors a vector-valued martingale. Averaged smoothness and the $L^2$ maximal inequality give
\begin{equation}\label{eq:upperdoob}
 \E e_{\mathcal E}^2\le\frac{4L^2}{b}\E V_{\mathcal E}.
\end{equation}
All query points are selected before their fresh samples are drawn. The horizon is fixed, so there is no stopping-time qualification in this inequality.

Write $W_t$ for the input potential of subproblem $t$ and $P_t$ for its output potential. Summing the movement bound of Lemma~\ref{lem:upperinner} within an epoch and absorbing the error term in \eqref{eq:upperdoob}, provided $b\ge8C_0qM$, gives
\begin{equation}\label{eq:upperepocherror}
 \E e_{\mathcal E}^2\le\frac{8C_0L}{b}
      \sum_{t\in\mathcal E}\E W_t.
\end{equation}
Consequently,
\begin{equation}\label{eq:uppersumP}
 \sum_t\E P_t
 \le\left(\sigma+\frac{136C_0\kappa q}{b}\right)\sum_t\E W_t
 \le\frac1{1024}\sum_t\E W_t,
\end{equation}
where one valid choice is
\begin{equation}\label{eq:upperbatch}
 b=\lceil\beta\kappa q\rceil,\qquad \beta=300000C_0.
\end{equation}
Indeed, $M\le(4\log2048+1)\kappa$ ensures $b\ge8C_0qM$. The choices $\sigma=1/2048$ and $\beta=300000C_0$ also give
\[
 \sigma+\frac{136C_0}{\beta}
 =\frac1{2048}+\frac{136}{300000}<\frac1{1024}.
\]
These constants are sufficient, not optimized. If an epoch uses $b>n$, every sampled difference may instead be replaced by the exact full-gradient difference, at cost $2n\le2b$ per update. A full-gradient difference introduces no new sampling error; because the epoch starts with an exact gradient and uses exact differences throughout, all its estimates remain exact. The same option applies to initialization and refinement with batch size $b_0$. It preserves the query bound and zero-respecting property.

There is also a useful subroutine with a fixed center: take $q=1$, reset the estimator, and perform one inner call. Equations~\eqref{eq:upperdoob}--\eqref{eq:uppersumP} show that its conditional expected output potential is at most half its input potential. Each such call costs $O(n+\kappa^2)$ IFO queries. Repeating $R$ calls reduces the expected potential by $2^{-R}$. This will be used only for initialization and the final output.

\subsection{Outer descent, initialization, and stationarity}
Define the envelope and proximal map
\begin{equation}\label{eq:upperenvelope}
 H(a)=\min_x\{\Phi(x)+L\norm{x-a}^2\},\qquad p(a)=x_a^*.
\end{equation}
Since $\Phi$ is $L$-weakly convex, $H$ is differentiable, $\nabla H(a)=2L(a-p(a))$, and $\nabla H$ is $2L$-Lipschitz. To see the last constant directly, weak monotonicity of $\nabla\Phi$ implies $\langle a-a',p(a)-p(a')\rangle\ge\norm{p(a)-p(a')}^2/2$, hence $I-p$ is nonexpansive. Moreover, $H(0)-\inf H\le\Delta$.

Initialize by approximately solving the subproblem centered at zero. Exact dual initialization gives $\nabla_x f(0,0)=\nabla\Phi(0)$. Strong convexity of $q_0(\cdot,0)$ and smoothness of $\Phi$ imply
\[
 W_0(0,0)\le\frac{3\norm{\nabla\Phi(0)}^2}{2L}
 \le3(1+\kappa)\Delta.
\]
Use the fixed-center iteration for $R_0=\lceil\log_2(384(1+\kappa))\rceil$ calls. Its output $(x_1,y_1)$ has expected potential at most $\Delta/128$ for the center zero. Also $L\norm{p(0)}^2\le\Delta$. Writing $P_0=W_0(x_1,y_1)$, the definition of $H$ and \eqref{eq:upperdistance} give directly
\[
 H(x_1)\le\Phi(p(0))+L\norm{x_1-p(0)}^2
 \le\Phi(0)+4P_0.
\]
Since $\inf H=\inf\Phi$, the expected envelope gap is at most $\Delta+\Delta/32$. Together with Lemma~\ref{lem:upperwarm}, this yields
\begin{equation}\label{eq:upperinit}
 \E[H(x_1)-\inf H]\le2\Delta,\qquad
 \E W_{x_1}(x_1,y_1)\le25\Delta.
\end{equation}
For example, $\E L\norm{x_1}^2\le2\Delta+8(\Delta/128)$, and the latter potential is at most $2(\Delta/128)+12L\E\norm{x_1}^2<25\Delta$.

For $t=1,\ldots,T$, run one inner call for $q_{x_t}$, starting at $(x_t,y_t)$, to obtain $(x_{t+1},y_{t+1})$. Reuse the estimator within each epoch as above. Put $d_t=x_{t+1}-x_t$. Smoothness of $H$ and \eqref{eq:upperdistance} give the pathwise inequalities
\begin{equation}\label{eq:upperouter}
 \begin{aligned}
 H(x_{t+1})&\le H(x_t)-\frac L2\norm{d_t}^2+8P_t,\\
 \norm{\nabla H(x_t)}^2&\le8L^2\norm{d_t}^2+32LP_t,\\
 W_{t+1}&\le2P_t+12L\norm{d_t}^2.
 \end{aligned}
\end{equation}
For the first inequality, write $x_{t+1}=p(x_t)+u_t$. The smooth descent inequality gives $H(x_{t+1})\le H(x_t)-L\norm{p(x_t)-x_t}^2+L\norm{u_t}^2$; then use $\norm{u_t}^2\le4P_t/L$.

Let $S_P=\sum_{t=1}^T\E P_t$, $S_W=\sum_{t=1}^T\E W_t$, and $S_d=\sum_{t=1}^T\E\norm{d_t}^2$. Equation~\eqref{eq:uppersumP} and the third inequality in \eqref{eq:upperouter} imply
\[
 S_P\le\frac{25\Delta+12LS_d}{1022}.
\]
Summing the first inequality in \eqref{eq:upperouter} and using \eqref{eq:upperinit} now yields
\begin{equation}\label{eq:upperglobal}
 LS_d\le6\Delta,\qquad S_P\le\Delta/8,\qquad S_W\le100\Delta,
 \qquad \sum_{t=1}^T\E\norm{\nabla H(x_t)}^2\le52L\Delta.
\end{equation}
Choose $T=\max\{1,\lceil256L\Delta/\eps^2\rceil\}$ and select $J$ uniformly from $[T]$, independently of the iterates. Then $\E\norm{\nabla H(x_J)}^2\le\eps^2/4$.

The final output must satisfy the criterion for $\Phi$, rather than just its envelope. Starting at the stored pair $(x_J,y_J)$, apply the fixed-center iteration to $q_{x_J}$ for
\[
 R_1=\max\left\{1,\left\lceil\log_2\!\left(\frac{1600L(1+\kappa)^2\Delta}{\eps^2}\right)\right\rceil\right\}
\]
calls. Since $\E W_{x_J}(x_J,y_J)=S_W/T\le100\Delta$, the resulting potential $P_{\mathrm{fin}}$ satisfies $\E P_{\mathrm{fin}}\le\eps^2/[16L(1+\kappa)^2]$. Return its primal coordinate $\hat x$. Using $L(1+\kappa)$-smoothness of $\Phi$, \eqref{eq:upperdistance}, and $\nabla\Phi(p(x_J))=\nabla H(x_J)$ gives
\[
 \E\norm{\nabla\Phi(\hat x)}^2
 \le2\E\norm{\nabla H(x_J)}^2
       +8L(1+\kappa)^2\E P_{\mathrm{fin}}\le\eps^2.
\]
Thus the argument proves the expected squared-gradient guarantee directly.

\subsection{Oracle cost and algorithm class}
For $1\le\kappa\le\sqrt n$, choose
\[
 q=\max\{1,\lfloor\sqrt n/\kappa\rfloor\},\qquad b=\lceil\beta\kappa q\rceil.
\]
An epoch has one full-gradient computation and $O(\kappa)$ estimator evaluations per subproblem, each requiring at most $2b$ component queries. The main iterations therefore cost
\[
 O\!\left(n+T\left(\frac nq+\kappa^2q\right)\right)
 =O(n+\sqrt n\,\kappa L\Delta\eps^{-2}).
\]
Here $q\ge\sqrt n/(2\kappa)$ and $\sqrt n\,\kappa\le n$ absorb rounding and the case $T=1$. Initialization and the final refinement cost $O((n+\kappa^2)(R_0+R_1))=\tO(n)$ in this regime. This proves \eqref{eq:smallupper}.

For $\kappa>\sqrt n$, all contraction, error, and stationarity arguments above still apply. Now $q=1$, so the main phase costs $O(n+T(n+\kappa^2))$, and initialization and final refinement cost $\tO(n+\kappa^2)$. Together they give $\tO((n+\kappa^2)(1+L\Delta\eps^{-2}))$, as stated in Theorem~\ref{thm:smallupper}.

Every step is a linear combination of previously queried points and component-gradient responses. The regularization center and any stored output pair are previously generated points. Hence the method is zero-respecting, including initialization and final refinement. Its number of queries is bounded deterministically by the displayed counts.\hfill$\square$

\section{The Imported Catalyst Guarantee}\label{app:catalyst}
We use \citet[Theorem 4.1, Equation (25)]{zhang} with the stopping tolerances of their Theorems~4.1--4.2. For the uniformly selected outer iterate, their guarantee is
\begin{equation}\label{eq:catalystcriterion}
 \E\norm{\nabla\Phi(\hat x_T)}^2
 \le\frac{268L}{5T}\Delta+\frac{28L}{5T}D_y^0,
 \qquad D_y^0=\norm{y_0-y^*(x_0)}^2.
\end{equation}
Thus the imported criterion is the expected squared gradient in \eqref{eq:criterion}. Corollary~4.1 and Section~4.2, Equation~(28), of that paper give the expected query bound
\[
 \tO\!\left((n+n^{3/4}\sqrt\kappa)L(\Delta+D_y^0)\eps^{-2}\right)
\]
for Catalyst with SVRG or SAGA under averaged smoothness. Our initialization makes $D_y^0=0$.

Choose SVRG as the inner method. Its iterates and variance-reduced estimates, and Catalyst's extrapolated points, are linear combinations of earlier points and component-gradient responses. The quadratic regularizers use previously formed centers. Starting from the origin, induction therefore gives Definition~\ref{def:zr}. Full-gradient and residual evaluations are implemented by $n$ component calls and are included in the cited query accounting. Consequently this is an admissible expected-query upper bound for Theorem~\ref{thm:matching}; no conversion from an expected-norm guarantee or component proximal oracle is needed.

\end{document}